\documentclass[11pt,reqno]{amsart}
\usepackage{fullpage}
\usepackage{a4wide}
\usepackage[utf8x]{inputenc}
\usepackage[T1]{fontenc}
\usepackage[english]{babel}
\usepackage{hyperref}
\usepackage{amsfonts,amsmath,amssymb,amsthm,amsrefs}
\usepackage{latexsym}
\usepackage{layout}
\usepackage{dsfont}
\usepackage{tikz}
\usepackage{parskip}

\newtheorem{theorem}{Theorem}
\newtheorem{prop}[theorem]{Proposition}
\newtheorem{lemma}[theorem]{Lemma}
\newtheorem{corol}[theorem]{Corollary}

\theoremstyle{definition}

\theoremstyle{remark}
\newtheorem{remark}[theorem]{Remark}

\newcommand{\p}{\mathbb{P}}
\newcommand{\e}{\mathbb{E}}

\newcommand{\reals}{\mathbb{R}}
\newcommand{\ind}{\mathbf{1}}

\newcommand{\me}{\mathrm{e}}
\newcommand{\md}{\mathrm{d}}

\newcommand{\Occ}{\mathcal{O}^{X}}

\newcommand{\drift}{c}

\def\beq{\begin{eqnarray}} \def\eeq{\end{eqnarray}}

\def\al*#1{\begin{align*}#1\end{align*}}

\def\ga*#1{\begin{gather*}#1\end{gather*}}

\def\alat*#1#2{\begin{alignat*}{#1}#2\end{alignat*}}
\def\bea{\begin{eqnarray*}}
\def\eea{\end{eqnarray*}}
\def\ml*#1{\begin{multline*}#1\end{multline*}}

\allowdisplaybreaks

\begin{document}
\title[]{Poissonian occupation times of spectrally negative Lévy processes with applications}
\author[]{Mohamed Amine Lkabous}
\address{Department of Statistics and Actuarial Science, University of Waterloo, Waterloo, ON, N2L 3G1, Canada}
\email{mohamed.amine.lkabous@uwaterloo.ca}
\date{\today }
\begin{abstract}
In this paper, we introduce the concept of \emph{Poissonian occupation times} below level $0$ of spectrally negative Lévy processes. In this case, occupation time is accumulated only when the process is observed to be negative at arrival epochs of an independent Poisson process. Our results extend some well known continuously observed quantities involving occupation times of spectrally negative Lévy processes. As an application, we establish a link between Poissonian occupation times and insurance risk models with Parisian implementation delays.
\end{abstract}

\maketitle

\section{Introduction}
Let $X=\left(X_t\right)_{t\geq 0}$ be a spectrally negative Lévy process. We define the first passage time above $b$,
\begin{align*}
\tau_b^+ = \inf\{t>0 \colon X_t > b\},
\end{align*}
with the convention $\inf \emptyset=\infty$.
Furthermore, let $T_i$ be the the arrival times of an independent Poisson process with intensity $\lambda>0$ and independent of $X$. 
In this paper, we define the Poissonian occupation time below $0$ over a finite-time horizon, by 
\begin{equation}
\Occ_{t,\lambda}= \sum_{n\in \mathbb{N}}(\tau^{+}_0 \circ \theta_{T_n})\ind_{\left\lbrace X_{T_n}<0,T_n< t \right\rbrace},
\end{equation}
where $\theta $ is the Markov shift operator ($X_{s}\circ \theta_{t}=X_{s+t} 
$). More precisely, occupation time is accumulated once the process is observed to be negative at Poisson arrival times (see Figure~\ref{fig2}). We will first study the joint Laplace transform of $\left(\tau_{b}^{+},\Occ_{\tau_{b}^{+},\lambda }\right)$ and, as a consequence, we examine the Laplace transform as well as the distribution of Poissonian occupation time over an infinite-time horizon, that is, 
$$\Occ_{\infty,\lambda}= \sum_{n\in \mathbb{N}}(\tau^{+}_0 \circ \theta_{T_n})\ind_{\left\lbrace X_{T_n}<0 \right\rbrace}.
$$
The analysis of quantities involving the duration of negative surplus (also called the \emph{time in red}) under continuous monitoring $ \Occ_{t} = \int_{0}^{t} \mathbf{1}_{\left(-\infty,0\right)}\left(X_s \right) \mathrm{d}s$ has been well studied in the literature. The Laplace transform of $\Occ_{\infty} = \int_{0}^{\infty} \mathbf{1}_{\left(-\infty,0\right)}\left(X_s \right) \mathrm{d}s$ was first derived by Landriault et al. \cite{landriaultetal2011}. Loeffen et al. \cite{loeffenetal2014} derived the joint Laplace transform of $\left(\tau^{+}_b,\Occ_{\tau^{+}_b}\right)$, generalizing the results in \cite{landriaultetal2011} and \cite{loeffenetal2014}. For a more general treatment, Li and Palmowski \cite{li-palmowski2017} studied weighted occupation times.  
Recently, Landriault et al. \cite{lietal2019} obtained an analytical expression for the distribution of the occupation time below level $0$ up to an (independent) exponential horizon for spectrally negative L\'{e}vy risk processes and refracted spectrally negative L\'{e}vy risk processes. 

In the aforementioned works, the study of continuously observed identities involving occupation times became challenging when surplus process has paths of unbounded variation. In \cite{landriaultetal2011}, the \textit{$\epsilon$-approximation} approach, which consists of a spatial shift of the sample paths of the underlying process, was used to avoid the problem caused by the infinity activity of the process and such that classical conditioning continue to hold. A second approach consists of introducing a sequence of bounded variation processes, $(X_n)_{n\geq 1}$, which converge to the unbounded variation process $X$ as $n$ goes to infinity (see Loeffen et al. \cite{loeffenetal2014}, Guérin and Renaud \cite{guerin_renaud_2015}). An important contribution of the Poisson observation approach is the unified proof for processes with bounded or unbounded variation paths, whereas these two cases need to be treated separately using the approximation approaches explained above and, when the Poisson observation rate goes to infinity, we recover results for occupation times under continuous monitoring (see e.g. Albrecher et al. \cite{albrecheretal2016} and Li et al. \cite{lietal2015}).
\subsection{Motivation}
In actuarial mathematics, occupation times can naturally be used as a measure of the risk inherent to an insurance portfolio. For instance, the time spent below a solvency threshold helps quantify the risk related to an insurance surplus process. The analysis of the duration of the negative surplus is also related to Parisian ruin models in which insurers are granted a period of time to re-emerge above the threshold level before ruin is deemed to
occur (see e.g. \cites{loeffenetal2013,loeffenetal2017}) and (\cites{landriaultetal2011,landriaultetal2014}). Two types of Parisian ruin are strongly related to $\Occ_t$: cumulative Parisian ruin and Parisian ruin with exponential delays. In the first case, ruin occurs when the surplus process stays cumulatively below a critical level longer than a pre-determined time period. Guérin and Renaud \cite{guerinrenaud2015} studied the probability of cumulative Parisian ruin for the Cramér–Lundberg risk model with exponential claims by giving an explicit representation for the distribution of $\Occ_t$. A general treatment for Lévy risk processes has been recently studied by Landriault et al. \cite{lietal2019}.
On the other hand, the probability of Parisian ruin with exponential delays was first studied in \cite{landriaultetal2011} through the following connection between the occupation time $\Occ_{\infty}$ and this type of Parisian ruin, for $x\in \reals$, 
\begin{equation}\label{link}
\p_{x}\left( \rho_\lambda<\infty\right)=1-\e_{x}\left[ \me^{-\lambda\Occ_{\infty}} \right], 
\end{equation}
where
\begin{equation}\label{kappalam}
\rho_\lambda=\inf \left\{ t>0\text{ | }t-g_{t}>\me^{g_{t}}_{\lambda} \right\},
\end{equation}
and $g_t = \sup \left\lbrace 0 \leq s \leq t \colon X_s \geq 0\right\rbrace$, while $\me^{g_{t}}_{\lambda}$ is an exponentially distributed random variable with rate $\lambda>0$ and independent of $X$ (see also Baurdoux et al. \cite{baurdoux_et_al_2015}). More precisely, a copy of $\me_\lambda$ is assigned to each excursion below level $0$ of the process. Also, it is known that Parisian ruin time in \eqref{kappalam} corresponds to the first passage time when $X$ is observed below $0$ at Poisson arrival times, that is
 \begin{align}\label{T}
T_{0}^- &= \min\{T_{i} \colon X_{T_{i}}<0, \textit{ }i\in \mathbb{N}\}, 
\end{align}
where $T_{i}$ are the arrival times of an independent Poisson process of rate $\lambda>0$.

This paper is partly motivated by the study of Parisian ruin with Erlang implementation delays. Such type of Parisian ruin has already been studied for Lévy insurance risk models (see, e.g. Albrecher and Ivanovs \cite{albrecher-ivanovs2017}, Landriault et al. \cite{landriaultetal2014} and Frostig and Keren-Pinhasik\cite{Frostig2019}). However, in contrast to Parisian ruin in \eqref{kappalam}, the relationship in Equation \eqref{link} does not hold any more when the delay follows an Erlang-distributed implementation delays. In this paper, we establish a link between Poissonian occupation time and Parisian ruin with Erlang-distributed implementation delays (see Section \ref{sectionexamp}). First, we study Parisian ruin time with implementation delays given by a sum of two independent exponential random variables with different rates. Thus, we extend the study to Parisian ruin with Erlang$(2,\lambda)$ implementation delays. We will also derive several fluctuation identities as well as the Gerber-Shiu distribution. For a more general case, namely Parisian ruin with Erlang$(n,\lambda)$ implementation delays, a recursive expression will be derived using the $n^{th}$-Poissonian occupation time.


The rest of the paper is organized as follows. In Section \ref{section1}, we present the necessary background material on spectrally negative L\'evy processes and scale functions. The main results, as well as the proofs, are presented in Section \ref{section3}. Applications to Parisian insurance risk models are presented in section \ref{sectionexamp}. In the Appendix, a few well known fluctuation identities with delays are presented.
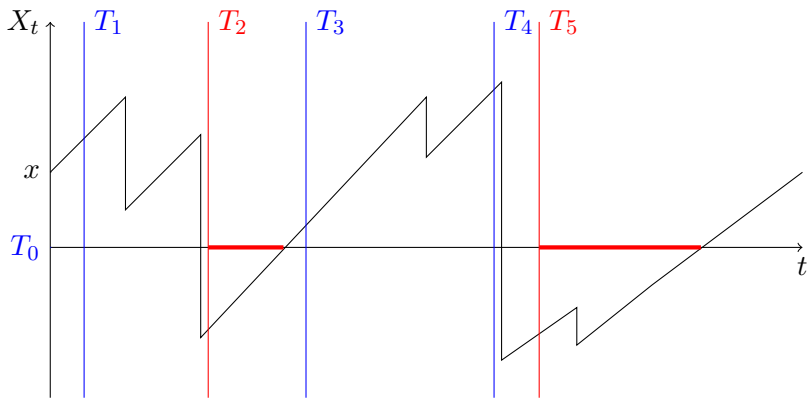
\begin{figure}[h!]
\begin{center}
\begin{tikzpicture}
  \draw[ thin,color=gray] (-0.1,-1.1)  (3.9,3.9);
\draw[->] (0,0) -- (10,0) node[below] {$t$};
 \draw[-] (0,1) -- (0,1) node[left] {$x$};
\draw[->] (0,-2) -- (0,3) node[left] {$X_t$};
\draw[color=black] (0,1) -- (1,2) --(1,0.5)--(2,1.5)--(2,-1.2)--(5,2)--(5,1.2)--(6,2.2)--(6,-1.5)--(7,-0.8)--(7,-1.3)--(8,-0.5)--(10,1);
   \draw[-][blue](0,0) -- (0,0) node[left] {$T_0$};
  \draw[-] [blue](0.45,-2) -- (0.45,3) node[right] {$T_{1}$};
\draw[-] [blue](3.4,-2) -- (3.4,3) node[right] {$T_3$};
   \draw[-] [red](2.1,-2) -- (2.1,3) node[right] {$T_2$};
    \draw[-] [ultra thick, red](2.1,0) -- (3.1,0) node[right] {};
      \draw[-] [blue](5.9,-2) -- (5.9,3) node[right] {$T_4$}; 
       \draw[-] [red](6.5,-2) -- (6.5,3) node[right] {$T_5$}; 
\draw[-] [ultra thick, red](6.5,0) -- (8.65,0) node[right] {};
\end{tikzpicture}
\caption{Illustration of Poissonian occupation time below $0$.}\label{fig2}
\end{center}
\end{figure}
\section{Preliminaries}\label{section1}
In this section, we present the necessary
background material on spectrally negative Lévy processes. 
\subsection{L\'{e}vy insurance risk processes}
 A L\'{e}vy insurance risk process $X$ is a process with stationary and independent increments and no positive jumps. To avoid trivialities, we exclude the case where $X$ has monotone paths. As the L\'{e}vy process $X$ has no positive jumps, its Laplace transform
exists: for all $\lambda, t \geq 0$, 
\begin{equation*}
\mathbb{E} \left[ \mathrm{e}^{\lambda X_t} \right] = \mathrm{e}^{t
\psi(\lambda)} ,
\end{equation*}
where 
\begin{equation*}
\psi(\lambda) = \gamma \lambda + \frac{1}{2} \sigma^2 \lambda^2 +
\int^{\infty}_0 \left( \mathrm{e}^{-\lambda z} - 1 + \lambda z \mathbf{1}%
_{(0,1]}(z) \right) \Pi(\mathrm{d}z) ,
\end{equation*}
for $\gamma \in \mathbb{R}$ and $\sigma \geq 0$, and where $\Pi$ is a $\sigma
$-finite measure on $(0,\infty)$ called the L\'{e}vy measure of $X$ such
that 	
\begin{equation*}
\int^{\infty}_0 (1 \wedge z^2) \Pi(\mathrm{d}z) < \infty .
\end{equation*}
We will use the standard Markovian notation: the law of $X$ when starting from $%
X_0 = x$ is denoted by $\mathbb{P}_x$ and the corresponding expectation by $%
\mathbb{E}_x$. We write $\mathbb{P}$ and $\mathbb{E}$ when $x=0$.\\
We recall the definition of standard first-passage time below level $b \in \mathbb{R}$,
\begin{align*}
\tau_b^- &= \inf\{t>0 \colon X_t<b\}.
\end{align*}
with the convention $\inf \emptyset=\infty$.
\subsection{Scale functions}

We now present the definition of the scale functions $W_{q}$ and $Z_{q}$ of $%
X$. First, recall that there exists a function $\Phi \colon [0,\infty) \to
[0,\infty)$ defined by $\Phi_{q} = \sup \{ \lambda \geq 0 \mid \psi(\lambda)
= q\}$ (the right-inverse of $\psi$) such that 
\begin{equation*}
\psi ( \Phi_{q} ) = q, \quad q \geq 0 .
\end{equation*}%
%
Now, for $q \geq 0$, the $q$-scale function of the process $X$ is defined as
the continuous function on $[0,\infty)$ with Laplace transform 
\begin{equation}  \label{def_scale}
\int_0^{\infty} \mathrm{e}^{- \lambda y} W_{q} (y) \mathrm{d}y = \frac{1}{%
\psi_q(\lambda)} , \quad \text{for $\lambda > \Phi_q$,}
\end{equation}
where $\psi_q(\lambda)=\psi(\lambda) - q$.  This function is unique,
positive and strictly increasing for $x\geq0$ and is further continuous for $%
q\geq0$. We extend $W_{q}$ to the whole real line by setting $W_{q}(x)=0$
for $x<0$. We write $W = W_{0}$ when $q=0$.\newline
We also define another scale function $Z_{q}(x,\theta )$ by 
\begin{equation}
Z_{q}(x,\theta )=\mathrm{e}^{\theta x}\left( 1-\psi_q (\theta ) \int_{0}^{x}%
\mathrm{e}^{-\theta y}W_{q}(y)\mathrm{d}y\right) ,\quad x\geq 0,
\end{equation}%
and $Z_{q}(x,\theta )=\mathrm{e}^{\theta x}$ for $x<0$. For $\theta=0$, we
get 
\begin{equation}  \label{eq:zqscale}
Z_{q}(x,0)=Z_{q}(x) = 1 + q \int_0^x W_{q}(y)\mathrm{d }y, \quad x \in 
\mathbb{R}.
\end{equation}
Using \eqref{def_scale}, we can also re-write the scale function $%
Z_{q}(x,\theta )$ as follows 
\begin{equation}  \label{Zv2}
Z_{q}(x,\theta )=\psi_q (\theta ) \int_{0}^{\infty}\mathrm{e}%
^{-\theta y}W_{q}(x+y)\mathrm{d}y ,\quad x\geq 0,\mathit{\ }\theta\geq
\Phi_q.
\end{equation}
It is known that
\begin{equation}\label{CM1}
\lim_{b \rightarrow \infty} \dfrac{W_{q}(x+b)}{W^{q}(b)} = \me^{\Phi_q x}.
\end{equation}
We also recall the \emph{second generation} scale function, that is, for $p,p+q\geq 0$ and $x \in \reals$, we have
\begin{align}\label{eq:convsnlp}
\mathcal{W}_{a}^{\left( p,q\right) }\left( x\right) &= W_{p}\left( x\right) +q\int_{a}^{x}W_{p+q}\left( x-y\right)W_{p}\left( y\right) \md y \nonumber\\ &= W_{p+q}\left( x\right)-q\int_{0}^{a}W_{p+q}\left( x-y\right) W_{p}\left( y\right) \md y,
\end{align}
where the second equation follows using the following identity obtained in \cite{loeffenetal2014} 
\begin{equation}\label{conveq}
(s-p) \int_0^x W_p(x-y) W_s(y) \mathrm{d}y = W_s(x) - W_p(x).
\end{equation}
We also have a slightly general case taken from \cite{MR3257360}: for $s\geq 0$, $x>a$, we have
\begin{equation}\label{convolution2}
(s-(p+q)) \int_a^x W_{s}(x-y)\mathcal{W}_{a}^{\left( p,q\right) }\left(y\right) \mathrm{d}y \\
=\mathcal{W}_{a}^{\left(p,s-p\right) }\left(x\right)-\mathcal{W}_{a}^{\left( p,q\right) }\left(x\right).
\end{equation}
For later use, note that we can show 
\begin{equation}\label{LapscriptW}
\int_{0}^{\infty} \me^{-\theta z} \mathcal{W}_{a}^{\left( p,s\right)} \left(a+z\right) \mathrm{d}z = \frac{Z_{p}\left( a,\theta\right)}{ \psi_{p+s}(\theta)}, \quad \theta >\Phi_{p+s} .
\end{equation}

For the sake of compactness of the results, we will use the following function defined in \cite{albrecher-ivanovs2017}: for $x\in \reals$ and $\alpha,\beta \geq 0$, we have
\begin{equation}\label{eq:ztilde}
\tilde{Z_q}\left(x,\alpha,\beta  \right)
=\dfrac{\psi_q(\alpha)Z_q(x,\beta)-\psi_q(\beta)Z_q(x,\alpha)}{\alpha-\beta},
\end{equation}
and when $\alpha=\beta$, we obtain 
\begin{equation}\label{eq:ztilde2}
\tilde{Z}_q\left(x,\alpha,\alpha  \right)
=\psi_q^{\prime}(\alpha)Z_q(x,\alpha)-\psi_q(\alpha)Z_q^{\prime}(x,\alpha),
\end{equation}
where $Z^{\prime}_q$ is the derivative of $Z_q$ taken with respect to the second argument. We write $\tilde{Z} = \tilde{Z}_{0}$ when $q=0$.\\
Finally, we recall Kendall's identity that provides the distribution of the first upward crossing of a specific level (see \cite[Corollary VII.3]{bertoin1996}): on $%
(0,\infty) \times (0,\infty)$, we have 
\begin{equation}  \label{eq:Kendall}
r \mathbb{P}(\tau_z^+ \in \mathrm{d}r) \mathrm{d}z = z \mathbb{P}(X_r \in 
\mathrm{d}z) \mathrm{d}r .
\end{equation}

We refer the reader to \cite{kyprianou2014} for more details on spectrally negative Lévy processes. More examples and numerical computations related to scale functions can be found in e.g. \cite{kuznetsovetal2012} and \cite{surya2008}.
\section{Main results}\label{section3}
We are now ready to state our main results. First, we provide an explicit expression for the joint Laplace transform of $\left(\tau_{b}^{+},\Occ_{\tau_{b}^{+},\lambda }\right)$.
\begin{theorem}\label{maintheotwosidedth}
For $p,q\geq 0$, $\lambda>0$ and $ x \leq b$,
\begin{equation}\label{mainrestwosided}
\mathbb{E}_{x}\left[ \mathrm{e}^{-q\tau_{b}^{+}-p\Occ_{\tau_{b}^{+},\lambda }}
\ind_{\left\lbrace \tau_{b}^{+}<\infty \right\rbrace}\right]=\dfrac{\tilde{Z}_q\left(x,\Phi_{\lambda+q},\Phi_{p+q}  \right)}{\tilde{Z}_q\left(b,\Phi_{\lambda+q},\Phi_{p+q}  \right)}.
\end{equation}
\end{theorem}
\begin{proof}
First, we set
\begin{equation*}
g\left( x\right) =\mathbb{E}_{x}\left[ \mathrm{e}^{-q\tau _{b}^{+}-p\Occ%
_{\tau _{b}^{+},\lambda }}\mathbf{1}_{\left\{ \tau _{b}^{+}<\infty \right\} }%
\right] .
\end{equation*}
For $x<0$, by the strong Markov property and the spectral negativity of $X$, 
\begin{equation}\label{xnegat}
g\left( x\right) =\mathbb{E}_{x}\left[ \mathrm{e}^{-\left( p+q\right) \tau
_{0}^{+}}\right] g\left( 0\right) .
\end{equation}%
For $0\leq x\leq b,$ again by the strong Markov property, 
\begin{equation}\label{xpositt}
g\left( x\right) =\mathbb{E}_{x}\left[ \mathrm{e}^{-q\tau _{b}^{+}}\mathbf{1}%
_{\left\{ \tau _{b}^{+}<T_{0}^{-}\right\} }\right] +\mathbb{E}_{x}\left[ 
\mathrm{e}^{-qT_{0}^{-}}g\left( X_{T_{0}^{-}}\right) \mathbf{1}_{\left\{
T_{0}^{-}<\tau _{b}^{+}\right\} }\right] .
\end{equation}
Hence, plugging \eqref{xnegat} in \eqref{xpositt} and using \eqref{jointLap1sd}, we deduce
\begin{equation}\label{alllllxx}
g\left( x\right) =\frac{Z_{q}\left( x,\Phi _{p+q}\right) }{Z_{q}\left(
b,\Phi _{p+q}\right) }+\mathbb{E}_{x}\left[ \mathrm{e}^{-qT_{0}^{-}+\Phi
_{p+q}X_{T_{0}^{-}}}\mathbf{1}_{\left\{ T_{0}^{-}<\tau _{b}^{+}\right\} }%
\right] g\left( 0\right)   
\end{equation}
For $x=0$, we have
\begin{eqnarray*}
g\left( 0\right)  &=&\frac{\mathbb{E}\left[ \mathrm{e}^{-q\tau _{b}^{+}}%
\mathbf{1}_{\left\{ \tau _{b}^{+}<T_{0}^{-}\right\} }\right] }{1-\mathbb{E}%
\left[ \mathrm{e}^{-qT_{0}^{-}+\Phi _{p+q}X_{T_{0}^{-}}}\mathbf{1}_{\left\{
T_{0}^{-}<\tau _{b}^{+}\right\} }\right] }=\frac{\frac{1}{Z_{q}\left( b,\Phi
_{\lambda +q}\right) }}{1-\dfrac{\lambda }{\lambda -p}\left( 1-\frac{%
Z_{q}\left( b,\Phi _{p+q}\right) }{Z_{q}\left( b,\Phi _{\lambda +q}\right) }%
\right) } \\
&=&\frac{\lambda -p}{\left( \lambda Z_{q}\left( b,\Phi _{p+q}\right)
-pZ_{q}\left( b,\Phi _{\lambda +q}\right) \right) }=\frac{\lambda -p}{\left(
\Phi _{\lambda +q}-\Phi _{p+q}\right) }\frac{1}{\tilde{Z}_{q}\left( b,\Phi
_{\lambda +q},\Phi _{p+q}\right) }
\end{eqnarray*}
Thus, replacing the expression of $g(0)$ in \eqref{alllllxx}, we finally have
\begin{eqnarray*}
g\left( x\right)  &=&\frac{Z_{q}\left( x,\Phi _{p+q}\right) }{Z_{q}\left(
b,\Phi _{p+q}\right) }+\left( \frac{\lambda -p}{\left( \Phi _{\lambda
+q}-\Phi _{p+q}\right) }\frac{1}{\tilde{Z}_{q}\left( b,\Phi _{\lambda
+q},\Phi _{p+q}\right) }\right)   \notag \\
&&\times \dfrac{\lambda }{\lambda -p}\left( \frac{Z_{q}\left( x,\Phi
_{p+q}\right) Z_{q}\left( b,\Phi _{\lambda +q}\right) -Z_{q}\left( x,\Phi
_{\lambda +q}\right) Z_{q}\left( b,\Phi _{p+q}\right) }{Z_{q}\left( b,\Phi
_{\lambda +q}\right) }\right)  \\
&=&\frac{\tilde{Z}_{q}\left( x,\Phi _{\lambda +q},\Phi _{p+q}\right) }{%
\tilde{Z}_{q}\left( b,\Phi _{\lambda +q},\Phi _{p+q}\right) }.
\end{eqnarray*}
%

\end{proof}
As an immediate consequence of Theorem \eqref{maintheotwosidedth}, we obtain the Laplace transform of $\Occ_{\infty ,\lambda }$.
\begin{corol}\label{Lapcor}
For $\lambda,p >0$, $\e[X_1]>0$ and $x \in \reals $, 
\begin{equation}\label{Lap}
\mathbb{E}_{x}\left[ \mathrm{e}^{-p\Occ_{\infty ,\lambda }}\right]=\mathbb{E}\left[ X_{1}\right] \frac{\Phi _{p}\Phi _{\lambda }}{\lambda p }\tilde{Z}\left(x,\Phi_{\lambda},\Phi_{p}  \right).
\end{equation}
\end{corol}
\begin{proof}
Using the fact that,
\begin{equation*}\lim_{b \rightarrow \infty}\dfrac{Z(b,\theta)}{W(b)}=\dfrac{\psi(\theta)}{\theta}, \quad \theta >0,
\end{equation*}
we have
\begin{equation*}
\lim_{b \rightarrow \infty}
\dfrac{\tilde{Z}\left(b,\Phi_{\lambda},\Phi_{p}  \right)}{W(b)}=\dfrac{\lambda p \left(\Phi_{\lambda}-\Phi_{p}\right)}{\Phi_{\lambda}\Phi_{p}},
\end{equation*}
and thus
\begin{equation*}
\lim_{b \rightarrow \infty} \dfrac{\tilde{Z}\left(x,\Phi_{\lambda},\Phi_{p}  \right)/W(b)}{\tilde{Z}\left(b,\Phi_{\lambda},\Phi_{p}  \right)/W(b)}=\mathbb{E}\left[ X_{1}\right] \frac{\Phi _{p}\Phi _{\lambda }}{\lambda p }\tilde{Z}\left(x,\Phi_{\lambda},\Phi_{p}  \right).
\end{equation*}
\end{proof}
In the next theorem, we give an explicit expression for the distribution of $ \Occ_{\infty ,\lambda }$.
\begin{theorem}
For $\lambda>0$, $r\geq 0$ and $x \in \reals $, we have
\begin{eqnarray}\label{delaydistr}
\mathbb{P}_{x}\left( \Occ_{\infty ,\lambda }\in \mathrm{d}r\right)  &=&%
\mathbb{E}\left[ X_{1}\right] \frac{\Phi
_{\lambda }}{\lambda }Z\left( x,\Phi _{\lambda }\right)\delta _{0}\left( \mathrm{d}r\right) +\mathbb{E}\left[
X_{1}\right]\frac{\Phi _{\lambda }^{2}}{%
\lambda }\Gamma_\lambda \left(r\right) Z\left( x,\Phi _{\lambda }\right) \mathrm{d}r \notag \\
&&-\mathbb{E}\left[ X_{1}\right] \Phi _{\lambda }\left( \Gamma_{\lambda } \left(r\right) W\left( x\right) +\int_{0}^{r}\Gamma _{\lambda }\left(r-s\right) \Lambda ^{\prime }\left( x,s\right) \mathrm{d}s\right)
\mathrm{d}r,
\end{eqnarray}
where 
$$\Gamma_\lambda (r)=\int^{\infty}_0 \me^{\Phi_\lambda z} \dfrac{z}{r}\p\left(X_r \in \md z\right),$$
and 
$$
\Lambda ^{\prime}\left( x,r\right) =\int_{0}^{\infty }W^{\prime}\left(x+z\right) \frac{z}{r}\mathbb{P}\left( X_{r}\in \mathrm{d}z\right).
$$
\end{theorem}
\begin{proof}
From \eqref{Lap}, we have 
\begin{eqnarray*}
\mathbb{E}_{x}\left[ \mathrm{e}^{-p\Occ_{\infty ,\lambda }}\right]  &=&%
\mathbb{E}\left[ X_{1}\right] \frac{\Phi _{p}\Phi _{\lambda }}{\lambda p}%
\left( \frac{pZ\left( x,\Phi _{\lambda }\right) -\lambda Z\left( x,\Phi
_{p}\right) }{\Phi _{p}-\Phi _{\lambda }}\right)   \\
&=&\mathbb{E}\left[ X_{1}\right] \frac{Z\left( x,\Phi _{\lambda }\right)
\Phi _{\lambda }}{\lambda }\left( 1+\frac{\Phi _{\lambda }}{\Phi _{p}-\Phi
_{\lambda }}\right) -\mathbb{E}\left[ X_{1}\right] \Phi _{\lambda }\frac{%
\Phi _{p}Z\left( x,\Phi _{p}\right) }{p\left( \Phi _{p}-\Phi _{\lambda
}\right) }.
\end{eqnarray*}
Using Kendall's identity and Tonelli's Theorem, we have 
\begin{equation}\label{id1G}
\int_{0}^{\infty }\me^{-pr}\Gamma_{\lambda }(r)\md r=\dfrac{1}{\Phi_p-\Phi_\lambda}, \quad p>\lambda.
\end{equation}
We also have the following identities from Landriault et al. \cite{lietal2019}, 
\begin{equation}\label{id1}
\frac{\Phi_{p}Z\left( x,\Phi_{p}\right) 
}{p}=\int_{0}^{\infty }\me^{-py}\left(\Lambda ^{\prime }\left( x,y\right) +W\left( x\right)\delta _{0}\left( \mathrm{d}y\right) \right) \mathrm{d}y,
\end{equation}
By Laplace inversion, we deduce
\begin{eqnarray*}
\mathbb{P}_{x}\left( \Occ_{\infty ,\lambda }\in \mathrm{d}r\right)  &=&%
\mathbb{E}\left[ X_{1}\right] Z\left( x,\Phi _{\lambda }\right) \frac{\Phi
_{\lambda }}{\lambda }\delta _{0}\left( \mathrm{d}r\right) +\mathbb{E}\left[
X_{1}\right] Z\left( x,\Phi _{\lambda }\right) \frac{\Phi _{\lambda }^{2}}{%
\lambda }\Gamma_{\lambda } \left(r\right) \mathrm{d}r \\
&&-\mathbb{E}\left[ X_{1}\right] \Phi _{\lambda }\left[ \Gamma_{\lambda } \left(r\right) W\left( x\right) +\int_{0}^{r}\Gamma _{\lambda }\left(r-s\right) \Lambda ^{\prime }\left( x,s\right) \mathrm{d}s\right]
\mathrm{d}r,
\end{eqnarray*}
\end{proof}
\begin{remark}
We can rewrite formula \eqref{delaydistr} as
\begin{equation*}
\mathbb{P}_{x}\left( \Occ_{\infty ,\lambda }\in \mathrm{d}r\right) =\mathbb{P%
}_{x}\left( T_{0}^{-}=\infty \right) \delta _{0}\left( \mathrm{d}r\right) +%
\mathbb{P}_{x}\left( \Occ_{\infty ,\lambda }\in \mathrm{d}r,T_{0}^{-}<\infty
\right),   \notag  
\end{equation*}
where 
\begin{equation*}
\mathbb{P}_{x}\left( T_{0}^{-}=\infty\right) =\mathbb{E}\left[ X_{1}\right] \frac{\Phi
_{\lambda }}{\lambda }Z\left( x,\Phi _{\lambda }\right) ,
\end{equation*}
and 
\begin{eqnarray*}
\mathbb{P}_{x}\left( \Occ_{\infty ,\lambda }\in \mathrm{d}r,T_{0}^{-}<\infty
\right)  &=&\mathbb{E}\left[ X_{1}\right] \frac{\Phi _{\lambda }\Gamma
_{\lambda }\left( r\right) }{\lambda }\left( \Phi _{\lambda }Z\left( x,\Phi _{\lambda }\right) -\lambda
W\left( x\right) \right) \mathrm{d}r  \notag  \label{delaydistr} \\
&&-\mathbb{E}\left[ X_{1}\right] \Phi _{\lambda }\left( \int_{0}^{r}\Gamma
_{\lambda }\left( r-s\right) \Lambda ^{\prime }\left( x,s\right) \mathrm{d}%
s\right) \mathrm{d}r.
\end{eqnarray*}
Under continuous monitoring, we have the following expression for the distribution of   $\Occ_{\infty}$
\begin{equation*}\label{districon}
\p_{x}\left( \Occ_{\infty} \in  \md r \right)= \e[X_{1}] \left(
W(x)\delta_0(\md r)+ \Lambda^{\prime}(x,r)\md r \right).
\end{equation*}
(See Landriault et al. \cite{lietal2019} for more results on the distribution of occupations times).
\end{remark}
\subsection{Discussion on the results}\label{sect:discussion}
Our fluctuation identities are arguably compact and have a similar structure as classical fluctuation identities (continuous monitoring). Indeed, from Theorem\eqref{maintheotwosidedth} we recover the identity in \cite[Corollary 2.(ii)]{loeffenetal2014} which is the joint Laplace transform of $\left(\tau_{b}^{+},\Occ_{\tau_{b}^{+}}\right)$ given by
$$\mathbb{E}_{x}\left[ \mathrm{e}^{-q\tau_{b}^{+}-p\Occ_{\tau_{b}^{+} }}
\ind_{\left\lbrace \tau_{b}^{+}<\infty \right\rbrace}\right]=\dfrac{Z_q\left(x,\Phi_{p+q}  \right)}{Z_q\left(b,\Phi_{p+q}  \right)},$$
which follows immediately from the fact that $\lim_{\lambda \rightarrow \infty}Z_q(b,\Phi_{\lambda+q})/\lambda=0.$ 
We also recover the Laplace transform of $\Occ_{\infty}$ by letting $\lambda\rightarrow \infty$ in Equation \eqref{Lap}. Indeed, we have
\begin{eqnarray}\label{rem1}
\mathbb{E}_{x}\left[ \mathrm{e}^{-p\Occ_{\infty }}\right]  &=&\lim_{\lambda
\rightarrow \infty }\mathbb{E}_{x}\left[ \mathrm{e}^{-p\Occ_{\infty ,\lambda
}}\right] =\lim_{\lambda \rightarrow \infty }\mathbb{E}\left[ X_{1}\right] 
\frac{\Phi _{p}\Phi _{\lambda }}{\lambda p}\tilde{Z}\left( x,\Phi _{\lambda
},\Phi _{p}\right) \notag \\
&=&\lim_{\lambda \rightarrow \infty }\mathbb{E}\left[ X_{1}\right] \frac{%
Z\left( x,\Phi _{\lambda }\right) \Phi _{\lambda }}{\lambda }\left( \frac{%
\Phi _{p}}{\Phi _{p}-\Phi _{\lambda }}\right) -\lim_{\lambda \rightarrow
\infty }\mathbb{E}\left[ X_{1}\right] \Phi _{\lambda }\frac{\Phi _{p}Z\left(
x,\Phi _{p}\right) }{p\left( \Phi _{p}-\Phi _{\lambda }\right) }\notag \\
&=&\mathbb{E}\left[ X_{1}\right] \frac{\Phi _{p}}{p}Z\left( x,\Phi
_{p}\right) ,
\end{eqnarray}
where in the last equality follows from the fact that $\Phi_{\lambda}\rightarrow \infty$ when $\lambda\rightarrow \infty$ and
$$\lim_{\lambda \rightarrow \infty}  \frac{Z\left( x,\Phi _{\lambda }\right)
\Phi _{\lambda }}{\lambda }=\lim_{\lambda \rightarrow \infty}\Phi _{\lambda } \int^{\infty}_{0}\me^{-\Phi _{\lambda }y}W(x+y)\md y=\lim_{y \rightarrow 0} W(x+y) =W(x).$$
which follows using the Initial Value Theorem of Laplace transform. \\
Also, given that $\tilde{Z}\left(x,\Phi_{\lambda},\Phi_{p}  \right)=\tilde{Z}\left(x,\Phi_{p},\Phi_{\lambda}  \right)$, we obtain
\begin{equation}\label{rem2}
\lim_{p \rightarrow \infty} \mathbb{E}\left[ X_{1}\right] \frac{\Phi _{p}\Phi _{\lambda }}{\lambda p }\tilde{Z}\left(x,\Phi_{\lambda},\Phi_{p}  \right)=\mathbb{E%
}\left[ X_{1}\right] \frac{\Phi _{\lambda}}{\lambda} Z\left( x,\Phi _{\lambda}\right).
\end{equation}
\begin{remark}
It is possible to extend our results to Poissonian occupation times of an arbitrary interval $(a,b)$. Indeed, let $\tau_{a,b} = \tau_b^+ \wedge \tau_a^- $, we define
\begin{equation}
\Occ_{t,\lambda}(a,b)= \sum_{n\in \mathbb{N}}( \tau_{a,b} \circ \theta_{T_n})\ind_{\left\lbrace a<X_{T_n}<b,T_n< t \right\rbrace}.
\end{equation}
This is left for future research.
\end{remark}
\section{Applications : Parisian ruin}\label{sectionexamp}
In this section, we want to provide a link between Poissonian occupation times and Parisian ruin models with Erlang-distributed implementation delays. 
\subsection{Parisian ruin with implementation delays given by a sum of two independent exponential random variables}\label{sect:subssectionexamp}
Under the Poissonian occupation time, occupation time is accumulated once the process is observed to be negative at Poisson arrival times. In other words, occupation time is accumulated when process stays below $0$ for a period of time equal to a copy of an  exponentially distributed r.v. $\me_\lambda$ with rate $\lambda>0$. Now, if we consider implementation clock given by an exponentially distributed r.v. $\me_p$ with rate $p>0$ (independent of $X$ and $\me_\lambda$), then $\p_{x}\left( \Occ_{\infty ,\lambda }> \me_p \right)$ corresponds to the probability of Parisian ruin with implementation delays modelled the sum of $\me_p$ and $\me_\lambda$. Hence, we have established the following connection
\begin{equation}\label{linkerlang}
\p_x\left(\rho^{(p,\lambda)}<\infty \right)=\p_{x}\left( \Occ_{\infty ,\lambda }>\me_p \right)=1-\mathbb{E}_{x}\left[ \mathrm{e}^{-p\Occ_{\infty \lambda }}\right],
\end{equation}
where $\rho^{(p,\lambda)}$ is the Parisian ruin time given by
\begin{equation}\label{ruinsum}
\rho^{(p,\lambda)}=\inf \left\{ t>0\text{ | }t-g_{t}>\me^{g_{t}}_{p}+\me^{g_{t}}_{\lambda} \right\},
\end{equation}
and for the finite-time case, 
\begin{equation*}
\p_x\left(\rho^{(p,\lambda)}\leq t \right)=1-\mathbb{E}_{x}\left[ \mathrm{e}^{-p\Occ_{t ,\lambda}}\right].
\end{equation*}
Thus, using the relationship in Equation \eqref{linkerlang} together with Corollary \eqref{Lapcor}, we obtain the following expression for $\p_{x}\left(\rho^{(p,\lambda)}<\infty \right).$
 \begin{corol}\label{parruinerrl}
 For $p,\lambda>0$, $x\in \reals$ and $\e \left[ X_1\right]>0$, 
 \begin{equation} \label{Parsum}
\p_{x}\left(\rho^{(p,\lambda)}<\infty \right) =1-\e\left[ X_{1}\right] \frac{\Phi_{\lambda}\Phi _{p}}{\lambda p}\tilde{Z}\left(x,\Phi_{\lambda},\Phi_{p}  \right).
\end{equation}
 \end{corol}
\begin{remark}
Using the discussion in subsection \eqref{sect:discussion}, one immediately recovers the expression for the probabilities of Parisian ruin with exponentially distributed implementation delays, 
$$\lim_{p \rightarrow \infty}\p_{x}\left(\rho^{(p,\lambda)}=\infty \right)=\mathbb{E}\left[ X_{1}\right] \frac{\Phi _{\lambda}}{\lambda}Z\left( x,\Phi_{\lambda}\right) =\p_{x}\left(\rho_{\lambda}=\infty \right),$$
$$\lim_{\lambda \rightarrow \infty}\p_{x}\left(\rho^{(p,\lambda)}=\infty \right)=\mathbb{E}\left[ X_{1}\right] \frac{\Phi _{p}}{p}Z\left( x,\Phi_{p}\right) =\p_{x}\left(\rho_{p}=\infty \right) ,$$
and finally, \eqref{Parsum} reduces to $\mathbb{P}\left(\tau _{0}^{-}<\infty \right) $ given that
\begin{eqnarray*}
\lim_{p\rightarrow \infty }\lim_{\lambda \rightarrow \infty }\mathbb{E}\left[
X_{1}\right] \frac{\Phi _{p}\Phi _{\lambda }}{\lambda p}\tilde{Z}\left(
x,\Phi _{\lambda },\Phi _{p}\right)  &=&\lim_{\lambda \rightarrow \infty
}\lim_{p\rightarrow \infty }\mathbb{E}\left[ X_{1}\right] \frac{\Phi
_{p}\Phi _{\lambda }}{\lambda p}\tilde{Z}\left( x,\Phi _{\lambda },\Phi
_{p}\right)   \label{rem3} \\
&=&\mathbb{E}\left[ X_{1}\right] W\left( x\right) =\mathbb{P}_{x}\left( \tau
_{0}^{-}<\infty \right) .
\end{eqnarray*}
\end{remark}
\begin{remark}
If we reverse the roles of $\me_{p}$ and $\me_{\lambda}$, that is, supposing arrival times of an independent Poisson process of rate $p>0$ and $\me_{\lambda}$ as an implementation clock, we also have
\begin{equation*}
\p_x\left(\rho^{(p,\lambda)}<\infty \right)=1-\mathbb{E}_{x}\left[ \mathrm{e}^{-\lambda\Occ_{\infty ,p}}\right].
\end{equation*}
\end{remark}
In the next theorem, we give further fluctuation identities involving the Parisian ruin time $\rho^{(p,\lambda)}$.
\begin{theorem}\label{forGs}
For $p,\lambda>0$, $b,q,\theta\geq 0$ and $x\leq b$, 
\begin{multline}\label{joinsum}
\mathbb{E}_{x}\left[ \mathrm{e}^{-q\rho ^{\left( p,\lambda \right) }+\theta
X_{\rho ^{\left( p,\lambda \right) }}}\mathbf{1}_{\left\{ \rho ^{\left(
p,\lambda \right) }<\tau _{b}^{+}\right\} }\right]\notag \\ =\frac{p}{\psi
_{q+\lambda }\left( \theta \right) \psi _{q+p}\left( \theta \right) }\left( 
\mathcal{E}^{\left( q,\lambda \right) }\left( x,\theta \right) -\frac{\tilde{%
Z}_{q}\left( x,\Phi _{\lambda +q},\Phi _{p+q}\right) }{\tilde{Z}_{q}\left(
b,\Phi _{\lambda +q},\Phi _{p+q}\right) }\mathcal{E}^{\left( q,\lambda
\right) }\left( b,\theta \right) \right) ,
\end{multline}
where 
\begin{equation*}
\mathcal{E}^{\left( q,\lambda \right) }\left( x,\theta \right) =\lambda
Z_{q}\left( x,\theta \right) -\psi _{q}\left( \theta \right) Z_{q}\left(
x,\Phi _{\lambda +q}\right).
\end{equation*}
For $x\in \reals$, 
\begin{multline}\label{twosidedsum3}
\e_{x}\left[\me^{-q\rho^{(p,\lambda)}+\theta X_{\rho^{(p,\lambda)}}}\ind_{\left\lbrace\rho^{(p,\lambda)}< \infty\right\rbrace} \right]\\=\frac{p}{\psi
_{q+\lambda }\left( \theta \right) \psi _{q+p}\left( \theta \right) }
\left(\mathcal{E}^{(q,\lambda)}(x,\theta) -\dfrac{\psi _{q}\left( \theta \right)(\Phi _{q+\lambda }-\theta )\left( \Phi _{p+q}-\Phi _{q}\right) }{p \left(\theta-\Phi_{q}\right)}\tilde{Z}_q\left(x,\Phi_{\lambda+q},\Phi_{p+q}  \right) \right).
\end{multline}
For $-a\leq x \leq b$,
\begin{equation}\label{mainrestwosidedlerl}
\mathbb{E}_{x}\left[ \mathrm{e}^{-q\tau_{b}^{+}}\ind_{\left\lbrace \tau_{b}^{+}<\rho^{(p,\lambda)}\wedge\tau_{-a}^{-} \right\rbrace}\right]=\dfrac{\tilde{W}_{q}^{\left( p,\lambda \right) }\left( x,a\right)  }{\tilde{W}_{q}^{\left( p,\lambda \right) }\left( b,a\right)  },
\end{equation}
where 
\begin{equation*}
\tilde{W}_{q}^{\left( p,\lambda \right) }\left( x,a\right) =\lambda \mathcal{W}_{x}^{\left(
q,p\right) }\left(x+a\right)W_{q+\lambda}(a)-p\mathcal{W}_{x}^{\left(q,\lambda\right) }\left(x+a\right)W_{p+q}(a).
\end{equation*}
When $a\rightarrow \infty$, we obtain
\begin{equation}\label{twosidedsum2}
\e_{x}\left[\me^{-q\tau^{+}_b}\ind_{\left\lbrace \tau^{+}_b<\rho^{(p,\lambda)}\right\rbrace} \right]=\dfrac{\tilde{Z}_q\left(x,\Phi_{\lambda+q},\Phi_{p+q}  \right)}{\tilde{Z}_q\left(b,\Phi_{\lambda+q},\Phi_{p+q}  \right)}.
\end{equation}
\end{theorem}
\begin{proof}
Let 
$$v\left(x\right)=\mathbb{E}_{x}\left[ \mathrm{e}^{-q\rho ^{\left( p,\lambda \right) }+\theta X_{\rho ^{\left( p,\lambda \right) }}}\mathbf{%
1}_{\left\{ \rho ^{\left( p,\lambda \right) }<\tau _{b}^{+}\right\} }\right].$$
For $x<0$, by the strong Markov property and the spectral negativity of $X$, we have 
\begin{equation}\label{xnegsum}
v\left(x\right)=\mathbb{E}_{x}\left[ \mathrm{e}^{-q\mathrm{e}_{p}+\theta X_{\me_p}}\mathbf{1}_{\left\{ \tau
_{0}^{+}>\mathrm{e}_{p}\right\} }\right] +\mathbb{E}_{x}\left[ \mathrm{e}%
^{-\left( p+q\right) \tau _{0}^{+}}\right] \mathbb{E}\left[ \mathrm{e}%
^{-q\rho ^{\left( p,\lambda \right) }}\mathbf{1}_{\left\{ \rho ^{\left(
p,\lambda \right) }<\tau _{b}^{+}\right\} }\right] .
\end{equation}%
For $0\leq x< b,$  using again the strong Markov property, we obtain 
\begin{equation}\label{xposum}
v\left(x\right)=\mathbb{E}_{x}\left[ \mathrm{e}^{-qT_{0}^{-}}v\left(X_{T_{0}^{-}} \right) \mathbf{1}%
_{\left\{ T_{0}^{-}<\tau^{+}_b \right\} }\right].
\end{equation}
Plugging \eqref{xnegsum} in \eqref{xposum}, we obtain, for all $x\in \reals$
\begin{eqnarray}\label{xalls}
v\left(x\right)
&=&\mathbb{E}_{x}\left[ \mathrm{e}^{-qT_{0}^{-}}\mathbb{E}_{X_{T_{0}^{-}}}%
\left[ \mathrm{e}^{-p\mathrm{e}_{q}+\theta X_{\me_p}}\mathbf{1}_{\left\{ \tau _{0}^{+}>%
\mathrm{e}_{q}\right\} }\right] \mathbf{1}_{\left\{ T_{0}^{-}<\tau
_{b}^{+}\right\} }\right]  \notag \\
&&+\mathbb{E}_{x}\left[ \mathrm{e}^{-qT_{0}^{-}}\mathbb{E}_{X_{T_{0}^{-}}}%
\left[ \mathrm{e}^{-\left( p+q\right) \tau _{0}^{+}}\right] \mathbf{1}%
_{\left\{ T_{0}^{-}<\tau _{b}^{+}\right\} }\right] \mathbb{E}\left[ \mathrm{e%
}^{-q\rho ^{\left( p,\lambda \right) }}\mathbf{1}_{\left\{ \rho ^{\left(
p,\lambda \right) }<\tau _{b}^{+}\right\} }\right]\notag \\
&=&\frac{-p}{\psi_{p+q}\left(\theta\right)}\mathbb{E}_{x}\left[ \mathrm{e}^{-qT_{0}^{-}+\theta X_{T_{0}^{-}}}\mathbf{1%
}_{\left\{ T_{0}^{-}<\tau _{b}^{+}\right\} }\right]\notag\\
&&+\frac{p}{\psi_{p+q}\left(\theta\right)}\mathbb{E}_{x}\left[ 
\mathrm{e}^{-qT_{0}^{-}}\mathrm{e}^{\Phi _{p+q}X_{T_{0}^{-}}}\mathbf{1}%
_{\left\{ T_{0}^{-}<\tau _{b}^{+}\right\} }\right]  \notag\\
&&+\mathbb{E}_{x}\left[ \mathbb{E}_{X_{T_{0}^{-}}}\left[ \mathrm{e}^{-\left(
p+q\right) \tau _{0}^{+}}\right] \mathbf{1}_{\left\{ T_{0}^{-}<\tau
_{b}^{+}\right\} }\right] \mathbb{E}\left[ \mathrm{e}^{-q\rho ^{\left(
p,\lambda \right) }}\mathbf{1}_{\left\{ \rho ^{\left( p,\lambda \right)
}<\tau _{b}^{+}\right\} }\right].
\end{eqnarray}%
For $x=0$ and using \eqref{jointLap1b}, we have 
\begin{eqnarray*}
v\left(0\right) &=&\frac{\frac{-p}{\psi_{p+q}\left(\theta\right)}\left( \mathbb{E}\left[ \mathrm{e}^{-qT_{0}^{-}+\theta X_{T_{0}^{-}}}%
\mathbf{1}_{\left\{ T_{0}^{-}<\tau _{b}^{+}\right\} }\right] -\mathbb{E}%
\left[ \mathrm{e}^{-qT_{0}^{-}+\Phi _{p+q}X_{T_{0}^{-}}}\mathbf{1}_{\left\{
T_{0}^{-}<\tau _{b}^{+}\right\} }\right] \right) }{1-\mathbb{E}\left[ 
\mathrm{e}^{-qT_{0}^{-}+\Phi _{p+q}X_{T_{0}^{-}}}\mathbf{1}_{\left\{
T_{0}^{-}<\tau _{b}^{+}\right\} }\right] } \\
&=&\frac{\frac{-p}{\psi_{p+q}\left(\theta\right))}\left( \dfrac{-\lambda }{\psi_{\lambda+q}\left(\theta\right)}\left( 1-\frac{Z_{q}\left(b,\theta \right) }{Z_{q}\left( b,\Phi _{q+\lambda }\right) }\right) -%
\dfrac{\lambda }{\lambda -p}\left( 1-\frac{Z_{q}\left( b,\Phi _{p+q}\right) 
}{Z_{q}\left( b,\Phi _{q+\lambda }\right) }\right) \right) }{1-\dfrac{%
\lambda }{\lambda -p}\left( 1-\frac{Z_{q}\left( b,\Phi _{p+q}\right) }{%
Z_{q}\left( b,\Phi _{q+\lambda }\right) }\right) } \\
&=&\frac{-p}{\psi_{p+q}\left(\theta\right)}\left( 1+\frac{\left( \lambda -p\right) }{\psi_{q+\lambda}\left(\theta\right) \left( \Phi _{\lambda +q}-\Phi _{p+q}\right) }\frac{ \mathcal{E}^{\left( q,\lambda \right) }\left( b,\theta \right) }{%
\tilde{Z}_{q}\left( b,\Phi _{\lambda +q},\Phi _{p+q}\right) }\right).
\end{eqnarray*}
Now, plugging the last expression in \eqref{xalls} combined with %
\eqref{jointLap1b}, we have, after few manipulations, 
\begin{eqnarray*}
v\left(x\right) 
&=&\frac{p}{\psi_{p+q}\left(\theta\right)}\dfrac{\lambda }{\psi_{\lambda +q}\left(\theta\right)}\left( Z_{q}\left( x,\theta\right)
-Z_{q}\left( x,\Phi _{\lambda +q}\right) \frac{Z_{q}\left( b,\theta\right) }{%
Z_{q}\left( b,\Phi _{\lambda +q}\right) }\right)  \\
&&+\frac{p}{\psi_{p+q}\left(\theta\right)}\dfrac{\lambda }{\psi_{\lambda +q}\left(\theta\right)}
\frac{Z_{q}\left( x,\Phi _{\lambda +q}\right) Z_{q}\left( b,\Phi
_{p+q}\right) }{Z_{q}\left( b,\Phi _{\lambda +q}\right) }  \\
&&\times \frac{\lambda
Z_{q}\left( b,\theta\right) }{\left( \Phi _{\lambda +q}-\Phi _{p+q}\right) \tilde{Z}%
_{q}\left( b,\Phi _{\lambda +q},\Phi _{p+q}\right) } \\
&&+\frac{p}{\psi_{p+q}\left(\theta\right)}\dfrac{\lambda }{\psi_{\lambda +q}\left(\theta\right)}
\frac{qZ_{q}\left( b,\Phi _{p+q}\right) Z_{q}\left( x,\Phi _{\lambda
+q}\right) }{\left( \Phi _{\lambda +q}-\Phi _{p+q}\right) \tilde{Z}%
_{q}\left( b,\Phi _{\lambda +q},\Phi _{p+q}\right) } \\
&=&\frac{p}{\psi_{p+q}\left(\theta\right)\psi_{\lambda +q}\left(\theta\right)}\left(
\lambda Z_q(x,\theta)-\psi_q\left(\theta\right) Z_q(x,\Phi_{\lambda+q})\right)
 \\
&&-\frac{p}{\psi_{p+q}\left(\theta\right)\psi_{\lambda +q}\left(\theta\right)}\dfrac{\tilde{Z}%
_{q}\left( x,\Phi _{\lambda +q},\Phi _{p+q}\right)}{\tilde{Z}%
_{q}\left( b,\Phi _{\lambda +q},\Phi _{p+q}\right)}\left(\lambda Z_q(b,\theta)-\psi_q\left(\theta\right) Z_q(b,\Phi_{\lambda+q})\right),
\end{eqnarray*}
which proves the first identity. To prove the second identity, we need compute the following limit
\begin{equation*}
\lim_{b\rightarrow \infty}\dfrac{\mathcal{E}^{(q,\lambda)}(b,\theta)}{\tilde{Z}_q\left(b,\Phi_{\lambda+q},\Phi_{p+q}  \right)}.
\end{equation*}
Given that
$$\lim_{b\rightarrow \infty}\dfrac{Z_q\left(b,\Phi_{\lambda+q} \right)}{Z_q\left(b  ,\theta \right)}=\dfrac{\lambda \left(\theta -\Phi_q\right)}{\psi_q\left(\theta\right)(\Phi_{\lambda+q}-\Phi_{q})} ,$$
and
$$\lim_{b\rightarrow \infty}\dfrac{\tilde{Z}_q\left(b,\Phi_{\lambda+q},\Phi_{p+q}  \right)}{Z_q\left(b,\theta  \right)}=\dfrac{\lambda p\left( \theta- \Phi_q\right)}{ \psi_q\left( \theta\right) (\Phi_{\lambda+q}-\Phi_{q})(\Phi_{p+q}-\Phi_{q})},$$
We obtain 
\begin{equation*}
\lim_{b\rightarrow \infty}\dfrac{\mathcal{E}^{(q,\lambda)}(b,\theta)}{\tilde{Z}_q\left(b,\Phi_{\lambda+q},\Phi_{p+q}  \right)}=\lim_{b\rightarrow \infty}\dfrac{\mathcal{E}^{(q,\lambda)}(b,\theta)/Z_q(b,\theta)}{\tilde{Z}_q\left(b,\Phi_{\lambda+q},\Phi_{p+q}  \right)/Z_q(b,\theta)}=\dfrac{\psi _{q}\left( \theta \right)(\Phi _{q+\lambda }-\theta )\left( \Phi _{p+q}-\Phi _{q}\right) }{p \left(\theta-\Phi_{q}\right)}.
\end{equation*}
Now, we prove identity in Equation \eqref{mainrestwosidedlerl}. We set
\begin{equation*}
w\left( x\right) =\e_{x}\left[\me^{-q\tau^{+}_b}\ind_{\left\lbrace \tau^{+}_b<\rho^{(p,\lambda)}\wedge \tau^{-}_{-a} \right\rbrace} \right]
\end{equation*}%
For $-a<x<0$, by the strong Markov property and the spectral negativity of $%
X$, we have 
\begin{equation}\label{xneg}
w\left( x\right) =\mathbb{E}_{x}\left[ \mathrm{e}^{-\left( p+q\right) \tau
_{0}^{+}}\mathbf{1}_{\left\{ \tau _{0}^{+}<\tau _{-a}^{-}\right\} }\right]
w\left( 0\right) =\frac{W_{p+q}\left( x+a\right) }{W_{p+q}\left( a\right) }%
w\left( 0\right).
\end{equation}%
For $0\leq x\leq b$, using again the strong Markov property, we obtain 
\begin{equation}\label{allx1}
w\left( x\right) =\mathbb{E}_{x}\left[ \mathrm{e}^{-q\tau _{b}^{+}}\mathbf{1}%
_{\left\{ \tau _{b}^{+}<T_{0}^{-}\wedge \tau_{-a}^{-} \right\} }\right] +\mathbb{E}_{x}\left[ 
\mathrm{e}^{-qT_{0}^{-}}w\left( X_{T_{0}^{-}}\right) \mathbf{1}_{\left\{
T_{0}^{-}<\tau _{b}^{+}\wedge \tau_{-a}^{-} \right\} }\right] .  
\end{equation}
Hence, plugging \eqref{xneg} in \eqref{allx1}, we deduce
\begin{equation*}
w\left( x\right) =\mathbb{E}_{x}\left[ \mathrm{e}^{-q\tau _{b}^{+}}\mathbf{1}%
_{\left\{ \tau _{b}^{+}<T_{0}^{-}\wedge \tau_{-a}^{-}  \right\} }\right] +\frac{w\left( 0\right) }{%
W_{p+q}\left( a\right) }\mathbb{E}_{x}\left[ \mathrm{e}^{-qT_{0}^{-}}W_{p+q}%
\left( X_{T_{0}^{-}}+a\right) \mathbf{1}_{\left\{ T_{0}^{-}<\tau
_{b}^{+} \wedge \tau_{-a}^{-} \right\} }\right] .
\end{equation*}%
For $x=0$, using \eqref{jointLap1sda} and  \eqref{LL1}, we have 
\begin{eqnarray*}
w\left( 0\right) &=&\frac{\mathbb{E}\left[ \mathrm{e}^{-q\tau _{b}^{+}}%
\mathbf{1}_{\left\{ \tau _{b}^{+}<T_{0}^{-}\wedge \tau_{-a}^{-} \right\} }\right] }{1-\mathbb{E}%
\left[ \mathrm{e}^{-qT_{0}^{-}}W_{p+q}\left( X_{T_{0}^{-}}+a\right) \mathbf{1%
}_{\left\{ T_{0}^{-}<\tau _{b}^{+}\wedge \tau_{-a}^{-} \right\} }\right]/W_{p+q}\left(a\right) } \\
&=&\frac{\left( \lambda -p\right)W_{p+q}\left(a\right)  W_{q+\lambda}\left(a\right)}{\lambda \tilde{W}^{(p,\lambda)}_{q}\left( b\right) }.
\end{eqnarray*}%
Now, plugging the last expression in \eqref{allx1}, we deduce 
\begin{eqnarray*}
w\left(x\right) &=&\frac{\mathcal{W}_{x}^{\left(q,\lambda\right) }\left(x+a\right) }{\mathcal{W}_{b}^{\left(q,\lambda\right) }\left(b+a\right) }+\frac{ W_{q+\lambda}\left(a\right)\left( \lambda -p\right) }{\lambda\tilde{W}^{(p,\lambda)}_{q}\left( b\right) }\dfrac{\lambda }{\lambda -p}\left( \mathcal{W}_{x}^{\left( q,p\right) }\left( x+a\right) -\mathcal{W}_{x}^{\left( q,\lambda \right) }\left( x+a\right) \right)  \\
&&-\frac{\left(\lambda-p\right)  W_{q+\lambda}\left(a\right)}{\lambda\tilde{W}^{(p,\lambda)}_{q}\left( b\right)} \dfrac{\lambda }{\lambda -p}\left( \mathcal{W}%
_{b}^{\left( q,p\right) }\left( b+a\right) -\mathcal{W}_{b}^{\left(
q,\lambda \right) }\left( b+a\right) \right) \frac{\mathcal{W}_{x}^{\left(q,\lambda\right) }\left(x+a\right) }{\mathcal{W}_{b}^{\left(q,\lambda\right) }\left(b+a\right)} \\
&=&\frac{\tilde{W}^{(p,\lambda)}_{q}\left( x\right) }{%
\tilde{W}^{(p,\lambda)}_{q}\left( b\right) },
\end{eqnarray*}%
which concludes the proof for the third identity.
Now, taking $a\rightarrow \infty$ in \eqref{mainrestwosidedlerl}, we have
\begin{equation*}
\lim_{a\rightarrow \infty }\dfrac{\tilde{W}_{q}^{\left( p,\lambda \right) }\left( x,a\right)}{\tilde{W}_{q}^{\left( p,\lambda \right) }\left( b,a\right)}=\lim_{a\rightarrow \infty }\frac{\tilde{W}_{q}^{\left( p,\lambda \right) }\left( x,a\right)/\left( W_{p+q}\left( a\right)W_{q+\lambda}\left( a\right)\right) }{\tilde{W}_{q}^{\left( p,\lambda \right) }\left( b,a\right)/\left( W_{p+q}\left( a\right)W_{q+\lambda}\left( a\right)\right) },
\end{equation*}
and
\begin{eqnarray*}
\lim_{a\rightarrow \infty }\dfrac{\tilde{W}_{q}^{\left( p,\lambda \right) }\left( x,a\right)}{W_{p+q}\left( a\right)W_{q+\lambda}\left( a\right)}&=&\lim_{a\rightarrow \infty }\dfrac{\lambda \mathcal{W}_{x}^{\left(q,p\right) }\left(x+a\right)W_{q+\lambda}(a)-p\mathcal{W}_{x}^{\left(q,\lambda\right) }\left(x+a\right)W_{p+q}(a)}{W_{p+q}\left( a\right)W_{q+\lambda}\left( a\right)}\\&=&\lambda Z_q\left(x,\Phi_{p+q}\right)-pZ_q\left(x,\Phi_{q+\lambda}  \right),
\end{eqnarray*}
where in the last equality we used the fact that, 
$$\lim_{a\rightarrow \infty } \dfrac{\mathcal{W}_{x}^{\left(q,p\right) }\left(x+a\right)}{W_{p+q}(a)}=\dfrac{W_{p+q}\left(x+a\right)-p\int_{0}^{x}W_{p+q}\left( x+a-y\right) W_{q}\left( y\right) \md y,}{W_{p+q}(a)}=Z_q\left(x,\Phi_{p+q}\right),$$
and similarly,
$$\lim_{a\rightarrow \infty } \dfrac{\mathcal{W}_{x}^{\left(q,\lambda\right) }\left(x+a\right)}{W_{q+\lambda}(a)}=Z_q\left(x,\Phi_{q+\lambda}\right),$$
\end{proof}
which both follow using \eqref{CM1}. The last identity can also be proved using the fact that
\begin{equation*}
\e_{x}\left[\me^{-q\tau^{+}_b}\ind_{\left\lbrace \tau^{+}_b<\rho^{(p,\lambda)} \right\rbrace} \right]=\mathbb{E}_{x}\left[ \mathrm{e}^{-q\tau_{b}^{+}-p\Occ_{\tau_{b}^{+},\lambda }}\ind_{\left\lbrace \tau_{b}^{+}<\infty \right\rbrace}\right].
\end{equation*}
 This ends the proof.
\begin{remark}
As it was shown in subsection \eqref{sect:discussion}, we have
\begin{equation*}
\lim_{p \rightarrow \infty}\dfrac{\tilde{Z}_q\left(x,\Phi_{\lambda+q},\Phi_{p+q}  \right)}{\tilde{Z}_q\left(b,\Phi_{\lambda+q},\Phi_{p+q}  \right)}=\dfrac{Z_q\left(x,\Phi_{\lambda+q}  \right)}{Z_q\left(b,\Phi_{\lambda+q}  \right)},
\end{equation*}
Thus, 
\begin{eqnarray*}
\lim_{p\rightarrow \infty }\mathbb{E}_{x}\left[ \mathrm{e}^{-q\rho ^{\left(
p,\lambda \right) }+\theta X_{\rho ^{\left( p,\lambda \right) }}}\mathbf{1}%
_{\left\{ \rho ^{\left( p,\lambda \right) }<\tau _{b}^{+}\right\} }\right] 
&=&\dfrac{\lambda }{\lambda -\psi _{q}(\theta )}\left( \mathcal{E}^{\left(
q,\lambda \right) }\left( x,\theta \right) -\dfrac{Z_q\left( x,\Phi _{\lambda+q
}\right) }{Z_q\left( b,\Phi _{\lambda+q }\right) }\mathcal{E}^{\left( q,\lambda
\right) }\left( b,\theta \right) \right)  \\
&=&\dfrac{\lambda }{\lambda -\psi _{q}(\theta )}\left( Z_{q}\left( x,\theta
\right) -Z_{q}\left( x,\Phi _{\lambda +q}\right) \dfrac{Z_{q}\left( b,\theta
\right) }{Z_{q}\left( b,\Phi _{\lambda +q}\right) }\right)  \\
&=&\mathbb{E}_{x}\left[ \mathrm{e}^{-qT_{0}^{-}+\theta X_{T_{0}^{-}}}\mathbf{%
1}_{\left\{ T_{0}^{-}<\tau _{b}^{+}\right\} }\right],
\end{eqnarray*}
which corresponds to identity \eqref{jointLap1b}. 
\end{remark}
Next is an expression of the Gerber–Shiu distribution at the Parisian ruin time $\rho ^{\left( p,\lambda \right) }$ .
\begin{theorem}\label{GSSTH}
For $p,\lambda>0$, $q\geq 0$, $x \leq b$ and $y\leq 0$,
\begin{equation}\label{GSS}
\mathbb{E}_{x}\left[ \mathrm{e}^{-q\rho ^{\left( p,\lambda \right)
}}\mathbf{%
1}_{\left\{X_{\rho ^{\left( p,\lambda \right) }}\in \mathrm{d}y\mathbf{,}\rho
^{\left( p,\lambda \right) }<\tau _{b}^{+}\right\} } \right] =p\left( \mathcal{E}_{y}^{\left( q,\lambda \right) }\left( x\right) -\frac{\tilde{Z}_{q}\left(
x,\Phi _{\lambda +q},\Phi _{p+q}\right) }{\tilde{Z}_{q}\left( b,\Phi
_{\lambda +q},\Phi _{p+q}\right) }\mathcal{E}_{y}^{\left( q,\lambda \right)
}\left( b \right) \right)\md y,
\end{equation}
where
\begin{eqnarray*}
\mathcal{E}_{y}^{\left( q,\lambda \right) }\left( x\right)  &=&\lambda
\frac{\mathcal{W}_{x}^{\left( q,\lambda \right) }\left( x-y\right)
-\mathcal{W}_{x}^{\left( q,p\right) }\left( x-y\right) }{\lambda -p} \\
&&-Z_{q}\left( x,\Phi _{\lambda +q}\right) \left(\dfrac{\lambda W_{q+\lambda }(-y) -pW_{p+q}\left( -y\right)}{\lambda-p}\right) .
\end{eqnarray*}
\end{theorem}
\begin{proof}
We have
\begin{equation}
\mathbb{E}_{x}\left[ \mathrm{e}^{-q\rho ^{\left( p,\lambda \right) }+\theta
X_{\rho ^{\left( p,\lambda \right) }}}\mathbf{1}_{\left\{ \rho ^{\left(
p,\lambda \right) }<\tau _{b}^{+}\right\} }\right]=\int^{0}_{-\infty}\me^{\theta y}  \mathbb{E}_{x}\left[ \mathrm{e}^{-q\rho ^{\left( p,\lambda \right)
}}\mathbf{%
1}_{\left\{X_{\rho ^{\left( p,\lambda \right) }}\in \mathrm{d}y\mathbf{,}\rho
^{\left( p,\lambda \right) }<\tau _{b}^{+}\right\} }\right].
\end{equation}
More precisely, we need to compute the Laplace inverse of $\frac{\mathcal{E}^{\left( q,\lambda \right) }\left( x,\theta \right)}{\psi_{q+\lambda }\left( \theta \right) \psi _{q+p}\left( \theta \right) }$ with respect to $\theta$. First, from \eqref{def_scale}, we have
\begin{eqnarray*}
\frac{\psi _{q}\left( \theta \right) }{\psi _{q+\lambda }\left( \theta
\right) \psi _{q+p}\left( \theta \right) } &=&\frac{1}{\psi _{q+p}\left(
\theta \right) }+\frac{\lambda }{\psi _{q+\lambda }\left( \theta \right)
\psi _{q+p}\left( \theta \right) } \\
&=&\int_{-\infty}^{0 }\mathrm{e}^{\theta y}\left( W_{p+q}\left( -y\right)
+\lambda \int_{0}^{-y}W_{q+\lambda }\left( -y-z\right) W_{p+q}\left(
z\right) \mathrm{d}z\right) \mathrm{d}y\\
&=&\int_{-\infty}^{0}\mathrm{e}^{\theta y}\left(\dfrac{\lambda W_{q+\lambda }(-y) -pW_{p+q}\left( -y\right)}{\lambda-p} \right) \mathrm{d}y,
\end{eqnarray*}
where the last equality follows using Equation \eqref{conveq}. Using \eqref{convolution2} and \eqref{LapscriptW}, we deduce
\begin{multline*}
\frac{Z_{q}\left( x,\theta \right) }{\psi _{q+\lambda }\left( \theta \right)
\psi _{q+p}\left( \theta \right) }=\int_{-\infty}^{0 }\mathrm{e}^{\theta
y}\left( \int_{0}^{-y}W_{q+\lambda }\left( -y-z\right) \mathcal{W}%
_{x}^{\left( q,p \right) }\left( x+z\right) \mathrm{d}z\right) \mathrm{d}y \\
=\int_{-\infty}^{0 }\mathrm{e}^{\theta
y}\left( \frac{\mathcal{W}_{x}^{\left( q,\lambda \right) }\left( x-y\right)
-\mathcal{W}_{x}^{\left( q,p\right) }\left( x-y\right) }{\lambda -p} \right) \mathrm{d}y.
\end{multline*}
The desired expression follows by Laplace inversion.
\end{proof}
\begin{remark}
The Gerber–Shiu distribution in Equation \eqref{GSS} has similar a structure as the one in \cite{baurdoux_et_al_2015} (although it is hard to show the convergence when $p\rightarrow \infty$), that is, 
\begin{eqnarray}\label{G-S-Parisian}
\e_{x}\left[\me^{-qT_{0}^{-}}
\mathbf{1}_{\left\{X_{T_{0}^{-}}\in \md y,T_{0}^{-}<\tau
_{b}^{+}\right\} } \right] =\lambda \left( \frac{Z_{q}\left( x,\Phi _{\lambda+q}\right) }{Z_{q}\left(b,\Phi _{\lambda+q}\right) }\mathcal{W}_{b}^{\left(
q,\lambda \right) }\left( b-y\right) -\mathcal{W}_{x}^{\left(q,\lambda
\right) }\left( x-y\right) \right)\md y.
\end{eqnarray}
\end{remark}
Using the relationship between the Parisian time $\rho^{(p,\lambda)}$ and Poissonian occupation time, we obtain the following expression for the Laplace transform of $\Occ_{\me_{q},\lambda}$ where $\me_{q}$ is an exponential random variable with rate $q>0$ that is independent of the process $X$.
\begin{corol}
For $p\geq 0$, $q,\lambda>0$ and $x \in \reals$, we have
\begin{equation}\label{uptoexp}
\mathbb{E}_{x}\left[\me^{-p\Occ_{\me_q,\lambda }}\right]=1-\frac{p}{\left( \lambda+q \right)\left(p+q \right)}
\left(\mathcal{E}^{(q,\lambda)}(x,0) -\dfrac{q\Phi _{q+\lambda }\left( \Phi _{p+q}-\Phi _{q}\right) }{p \Phi_{q}}\tilde{Z}_q\left(x,\Phi_{\lambda+q},\Phi_{p+q}  \right) \right).
\end{equation}
\end{corol}
\begin{proof}
First, we have
\begin{equation}\label{mainrestwosidedl33}
\mathbb{E}_{x}\left[\me^{-p\Occ_{\me_q,\lambda }}\right]= \mathbb{P}_{x}\left(\Occ_{\me_q,\lambda }<\me_p \right)
=\mathbb{P}_{x}\left(\rho ^{\left( p,\lambda \right) }>\me_q\right)=1-\mathbb{E}_{x}\left[\me^{-q\rho ^{\left( p,\lambda \right) }}\right].
\end{equation}
Then, using \eqref{twosidedsum3} for $\theta=0$, we obtain 
\begin{multline}\label{uptoexp}
\mathbb{E}_{x}\left[\me^{-p\Occ_{\me_q,\lambda }}\right]\\=1-\frac{p}{\left( \lambda+q \right)\left(p+q \right)}
\left(\mathcal{E}^{(q,\lambda)}(x,0) -\dfrac{q\Phi _{q+\lambda }\left( \Phi _{p+q}-\Phi _{q}\right) }{p \Phi_{q}}\tilde{Z}_q\left(x,\Phi_{\lambda+q},\Phi_{p+q}  \right) \right).
\end{multline}
\end{proof}
\begin{remark}
When $\lambda\rightarrow \infty$, we recover the Laplace transform of $\Occ_{\me_q}$,
\begin{equation}\label{uptoexp}
\mathbb{E}_{x}\left[\me^{-p\Occ_{\me_q }}\right]=\lim_{\lambda \rightarrow \infty} \mathbb{E}_{x}\left[\me^{-p\Occ_{\me_q,\lambda }}\right] \\=1-\frac{p}{\left(p+q \right)}
\left(Z_{q}(x) -\dfrac{q\left( \Phi _{p+q}-\Phi _{q}\right) }{p \Phi_{q}}Z_q\left(x,\Phi_{p+q} \right) \right),
\end{equation}
which can be found in \cite{lietal2019}.
\end{remark}
As an application of the previous results, we study Parisian ruin with Erlang$(2,\lambda)$ distributed implementation delays denoted by $\rho^{(2)}_\lambda$. This Parisian ruin time can also be expressed as follows
$$\rho^{(2)}_\lambda =T_{\tilde{N}^{-}_0},$$
where
$\tilde{N}^{-}_0=\min\left\lbrace i\in \mathbb{N} : \sup \left\lbrace X_t : t\in \left[ T_{i-1},T_i \right]\right\rbrace<0  \right\rbrace$ (see Albrecher and Ivanovs \cite{albrecher-ivanovs2017}).

Hence, letting $p \rightarrow \lambda$ in Corollary \eqref{parruinerrl}, we obtain the following expression for the probability of Parisian ruin with Erlang$(2,\lambda)$ implementation delays which has a similar structure as the one in Equation~\eqref{Pruine1}.
\begin{corol}
For $\lambda>0$, $x\in  \reals$ and $\e\left[ X_{1}\right]>0$, we have
 \begin{equation} \label{ParErlang}
\p_{x}\left(\rho^{(2)}_\lambda<\infty \right) =1-\e\left[ X_{1}\right] \frac{\Phi^{2} _{\lambda}}{\lambda^{2}}\tilde{Z}\left(x,\Phi_{\lambda},\Phi_{\lambda}  \right).
\end{equation}
 \end{corol}
Similarly, from Theorem \eqref{forGs}, we obtain the following results.
\begin{corol}
For $q,\lambda>0$, $x\leq b$ and $y\leq 0$, 
\begin{equation}\label{GSS2}
\mathbb{E}_{x}\left[ \mathrm{e}^{-q\rho ^{\left(2 \right)}_\lambda}\mathbf{1}_{\left\{X_{\rho ^{\left(2 \right) }_\lambda}\in \mathrm{d}y\mathbf{,}\rho^{\left(2 \right) }_\lambda<\tau _{b}^{+}\right\} }\right] =\lambda\left( \mathcal{E}_{y}^{\left(\lambda \right) }\left( x\right) -\frac{\tilde{Z}_{q}\left(
x,\Phi _{\lambda +q},\Phi _{\lambda+q}\right) }{\tilde{Z}_{q}\left( b,\Phi
_{\lambda +q},\Phi _{\lambda+q}\right) }\mathcal{E}_{y}^{\left(\lambda \right)
}\left( b \right) \right)\md y,
\end{equation}
where
\begin{eqnarray*}
\mathcal{E}_{y}^{\lambda }\left( x\right) =\lambda \frac{\partial \mathcal{W}_{x}^{\left( q,\lambda \right) }\left( x-y\right)}{\partial \lambda}-Z_{q}\left( x,\Phi _{\lambda +q}\right) \left( W_{q+\lambda }(-y) +\frac{\partial W_{q+\lambda }(-y)}{\partial \lambda} 
\right) .
\end{eqnarray*}
\begin{equation}\label{twosidederl}
\mathbb{E}_{x}\left[ \mathrm{e}^{-q\rho _{\lambda }^{\left( 2\right)}+\theta X_{\rho _{\lambda }^{\left( 2\right) }}}\mathbf{1}_{\left\{ \rho_{\lambda }^{(2)}<\tau^{+}_b \right\} }\right]=\dfrac{\lambda}{(\psi_{\lambda+q}\left(\theta\right))^{2}}
\left(\mathcal{E}^{(q,\lambda)}(x,\theta) -\dfrac{\tilde{Z}_q\left(x,\Phi_{\lambda+q},\Phi_{\lambda+q}  \right)}{\tilde{Z}_q\left(b,\Phi_{\lambda+q},\Phi_{\lambda+q}  \right)} \mathcal{E}^{(q,\lambda)}(b,\theta) \right),
\end{equation}
and
\begin{equation*}
\e_{x}\left[\me^{-q\tau^{+}_b}\ind_{\left\lbrace \tau^{+}_b<\rho^{(2)}_\lambda\right\rbrace} \right]=\dfrac{\tilde{Z}_q\left(x,\Phi_{\lambda+q},\Phi_{\lambda+q}  \right)}{\tilde{Z}_q\left(b,\Phi_{\lambda+q},\Phi_{\lambda+q}  \right)}.
\end{equation*}
For $x\in \reals,$
\begin{multline}\label{jointlaperl}
\mathbb{E}_{x}\left[ \mathrm{e}^{-q\rho _{\lambda }^{\left( 2\right)}+\theta X_{\rho _{\lambda }^{\left( 2\right) }}}\mathbf{1}_{\left\{ \rho_{\lambda }^{(2)}<\infty \right\} }\right] \\ =\frac{\lambda}{\left(\psi _{\lambda+q}\left( \theta \right) \right) ^{2}}\left(\mathcal{E}^{(q,\lambda)}(x)-\frac{\psi _{q}\left( \theta \right)(\Phi _{q+\lambda }-\theta) \left( \Phi _{\lambda+q}-\Phi _{q}\right) \tilde{Z}_{q}\left( x,\Phi _{\lambda +q},\Phi _{\lambda+q}\right) }{\lambda \left( \theta -\Phi _{q}\right) }\right).
\end{multline}
 \end{corol}


\begin{remark}
Setting $\theta=x=0$ in \eqref{jointlaperl}, we get 
 \begin{equation*}
\e\left[\me^{-q\rho^{(2)}_\lambda}\ind_{\left\lbrace \rho^{(2)}_\lambda<\infty \right\rbrace} \right]=\dfrac{\lambda}{\left(\lambda+q\right)}
-\dfrac{\lambda}{\left(\lambda+q\right)^{2}}\frac{\Phi _{\lambda+q}\left(\Phi _{\lambda+q}-\Phi _{q}\right)}{\Phi _{q}\Phi^{\prime}_{\lambda+q} },
\end{equation*}
which corresponds to Equation $(22)$ in \cite{albrecher-ivanovs2017} and $\Phi^{\prime}_{\lambda} $ is the derivative of $\Phi_{\lambda} $ with respect to the sub-index $\lambda$. 
\end{remark}
\subsubsection{Examples}
In this subsection, we compute the probability of Parisian ruin with Erlang$(2,\lambda)$ delays for the Brownian risk model suing the formula in~\eqref{ParErlang}.
\paragraph{\textbf{Brownian risk process}}
Let $X$ be a Brownian risk process, i.e.\
$$
X_t =x+ \mu t +\sigma B_t ,
$$
where $\mu>0$, $\mu>0$ and $B=\{B_t, t\geq 0\}$ is a standard Brownian motion. Then, we have 
$\psi (\theta )=\mu \theta +\frac{\sigma ^{2}}{2}\theta ^{2}$ and $\Phi _{\lambda }=\left( \sqrt{\mu ^{2}+2\sigma ^{2}\lambda }-\mu\right)\sigma ^{-2} .$
In this case, for $x\geq 0$ and $q>0$, the scale functions of $X$ are given by
\begin{align*}
W(x) &= \frac{1}{\mu} \left( 1-\mathrm{e}^{-2\mu x/\sigma^{2} } \right),
\end{align*} 
and
\begin{equation*}
Z\left( x,\theta \right) =\frac{\psi (\theta )}{\mu }\left( \frac{1}{\theta }%
-\frac{\mathrm{e}^{-2\mu \sigma ^{-2}x}}{\theta +2\mu \sigma ^{-2}}\right).
\end{equation*}
The derivative of $Z$ with respect to $\theta$ is given by
\begin{eqnarray*}
Z^{\prime }\left( x,\theta \right)  &=&\frac{\psi ^{\prime }(\theta )}{\mu }%
\left( \frac{1}{\theta }-\frac{\mathrm{e}^{-2\mu \sigma ^{-2}x}}{\theta
+2\mu \sigma ^{-2}}\right)  \\
&&+\frac{\psi (\theta )}{\mu }\left( \frac{\mathrm{e}^{-2\mu \sigma ^{-2}x}}{%
\left( \theta +2\mu \sigma ^{-2}\right) ^{2}}-\frac{1}{\theta ^{2}}\right) 
\\
&=&\psi ^{\prime }(\theta )\frac{Z\left( x,\theta \right) }{\psi (\theta )}+%
\frac{\psi (\theta )}{\mu }\left( \frac{\mathrm{e}^{-2\mu \sigma ^{-2}x}}{%
\left( \theta +2\mu \sigma ^{-2}\right) ^{2}}-\frac{1}{\theta ^{2}}\right).
\end{eqnarray*}
Using the expression in \eqref{eq:ztilde2}, we have
\begin{eqnarray*}
\tilde{Z}\left( x,\Phi _{\lambda },\Phi _{\lambda }\right)=\frac{\lambda }{\mu }\left( \frac{1}{\Phi _{\lambda }^{2}}-\frac{\mathrm{e%
}^{-2\mu \sigma ^{-2}x}}{\left( \Phi _{\lambda }+2\mu \sigma ^{-2}\right) }%
\right) .
\end{eqnarray*}
Putting all the terms together, we get the following expression for the probability of Parisian ruin with Erlang$(2,\lambda)$ delay,
\begin{equation*}
\mathbb{P}_{x}\left( \rho _{\lambda }^{(2)}<\infty \right) =1-\frac{
\left( \sqrt{\mu ^{2}+2\sigma ^{2}\lambda }-\mu \right) ^{2}}{\lambda
^{2}\sigma ^{4}}\left( \frac{1}{\Phi _{\lambda }}-\frac{\mathrm{e}%
^{-2\mu \sigma ^{-2}x}}{\left( \Phi _{\lambda }+2\mu \sigma ^{-2}\right) }%
\right) .
\end{equation*}
\paragraph{\textbf{Cramér-Lundberg process with exponential claims}}
Let $X$ be a Cramér-Lundberg risk processes with exponentially distributed claims, i.e.\
$$
X_t =x+ \drift t - \sum_{i=1}^{N_t} C_i ,
$$
where $N=\{N_t, t\geq 0\}$ is a Poisson process with intensity $\eta>0$, and where $\{C_1, C_2, \dots\}$ are independent and exponentially distributed random variables with parameter $\alpha$. The Poisson process and the random variables are mutually independent. Then, we have 
\begin{equation*}
\psi (\theta )=c\theta -\eta +\frac{\alpha \eta }{\theta +\alpha }\
\ \ \text{ and
\ }\Phi_{\lambda}= \frac{1}{2c}\left(\lambda+\eta -c\alpha +\sqrt{\left( \lambda+\eta -c\alpha \right) ^{2}+4c\alpha \lambda}\right).
\end{equation*}
The scale functions of $X$ are given by 
\begin{align*}
W(x) &= \frac{1}{c-\eta/\alpha} \left( 1- \frac{\eta}{c\alpha}\mathrm{e}^{(\frac{\eta}{c}-\alpha)x} \right),
\end{align*}
and 
\begin{equation*}
Z\left( x,\theta \right) =\frac{\psi (\theta )}{c-\eta /\alpha }\left( \frac{%
1}{\theta }-\frac{\eta }{c\alpha }\frac{\mathrm{e}^{\left( \eta /c-\alpha
\right) x}}{\theta +\alpha -\eta /c}\right) .
\end{equation*}
The derivative of $Z$ with respect to $\theta$ is given by
\begin{equation*}
Z^{\prime }\left( x,\theta \right) =\frac{\psi ^{\prime }(\theta )}{\psi
(\theta )}Z\left( x,\theta \right) +\frac{\psi (\theta )}{c-\eta /\alpha }%
\left( \frac{\eta }{c\alpha }\frac{\mathrm{e}^{\left( \eta /c-\alpha \right)
x}}{\left( \theta +\alpha -\eta /c\right) ^{2}}-\frac{1}{\theta ^{2}}\right),
\end{equation*}
and consequently,
\begin{equation*}
\tilde{Z}\left( x,\Phi _{\lambda },\Phi _{\lambda }\right) =\frac{\lambda }{%
c-\eta /\alpha }\left( \frac{1}{\Phi _{\lambda }^{2}}-\frac{\eta }{c\alpha }%
\frac{\mathrm{e}^{\left( \eta /c-\alpha \right) x}}{\left( \Phi _{\lambda
}+\alpha -\eta /c\right) ^{2}}\right) .
\end{equation*}
Putting all the pieces together, we obtain
\begin{equation*}
\mathbb{P}_{x}\left( \rho _{\lambda }^{(2)}<\infty \right) =1-\frac{1}{%
\lambda }\left( \frac{1}{\Phi _{\lambda }^{2}}-\frac{\eta }{c\alpha }\frac{%
\mathrm{e}^{\left( \eta /c-\alpha \right) x}}{\left( \Phi _{\lambda }+\alpha
-\eta /c\right) ^{2}}\right). 
\end{equation*}
\subsection{Parisian ruin with Erlang$(n,\lambda)$ implementation delays ($n\geq 3$)}
Now, we suppose that occupation time is accumulated after two consecutive observations of the surplus process below $0$, we denote the Poissonian occupation time by 
$$\Occ_{\infty,\lambda,2}= \sum_{n\in \mathbb{N}}(\tau^{+}_0 \circ \theta_{T_{n+1}})\ind_{\left\lbrace \sup_{t\in \left[ T_{n},T_{n+1}\right]} \left( X_t\right) <0 \right\rbrace} .$$
The Laplace transform of $\Occ_{\infty,\lambda,2}$ can be computed using the following a standard probabilistic decomposition 
\begin{eqnarray}\label{twostep}
\mathbb{E}_{x}\left[ \mathrm{e}^{-q\Occ_{\infty,\lambda,2 }}\right] 
=\p_x \left(\rho^{(2)}_\lambda=\infty \right)+\mathbb{E}_{x}\left[\mathrm{e}^{\Phi _{q}X_{\rho^{(2)}_\lambda}}\ind_{\left\{\rho^{(2)}_\lambda<\infty \right\} }\right] \mathbb{E}%
\left[ \mathrm{e}^{-q\Occ_{\infty,\lambda,2 }}\right] .
\end{eqnarray}

Letting $q\rightarrow \lambda$ in \eqref{twostep} combined with the results in the previous subsection \eqref{sect:subssectionexamp}, one can obtain an expression for probability of Parisian ruin with Erlang$(3,\lambda)$ implementation delays.\\

More generally, we denote the $n^{th}$-Poissonian occupation time by $\Occ_{t,\lambda,n}$. In this case, occupation time will be accumulated when the process $X$ has been observed to be negative at the last $n$ Poisson arrival times, i.e., when $\sup_{t\in \left[ T_{i},T_{i+n}\right]}(X_t)<0 $. Then, we have
$$\Occ_{\infty,\lambda,n}= \sum_{i\in \mathbb{N}}(\tau^{+}_0 \circ \theta_{T_{i+n}})\ind_{\left\lbrace \sup_{t\in \left[ T_{i},T_{i+n}\right]} \left( X_t\right) <0 \right\rbrace} .$$
Also, we denote $\rho^{(n)}_\lambda$ as the Parisian ruin time with Erlang$(n,\lambda)$ implementation delay, that is, 
$$\rho^{(n)}_\lambda=\inf \left\{ t>0\text{ | }t-g_{t}>T^{n,g_t}_{\lambda} \right\},$$
where $T^{n,g_t}_{\lambda}$ follows an Erlang$(n,\lambda)$ distribution. In particular, we have $\rho^{(0)}_\lambda = \tau^{-}_0$ and $\rho^{(1)}_\lambda = T^{-}_0$. From the discussion at the beginning of this section,  we have the following connection
\begin{equation}\label{genlink}
\mathbb{P}_{x}\left( \rho _{\lambda }^{(n)}<\infty \right)=1-\mathbb{E}_{x}\left[ \mathrm{e}^{-\lambda\Occ_{\infty ,\lambda,n }}\right].
\end{equation}
Thus, we obtain the following recursive formula for the probability of Parisian ruin with Erlang$(n,\lambda)$ implementation delays.
\begin{prop}\label{maintheo3}
For $n\in \mathbb{N}$, $\lambda>0$ and $x \in \reals $,
 \begin{equation}\label{decomp}
\mathbb{P}_{x}\left( \rho _{\lambda }^{(n)}=\infty \right)=\mathbb{P}_{x}\left( \rho _{\lambda }^{(n-1)}=\infty \right) +\mathbb{P}\left( \rho_{\lambda }^{(n-1)}=\infty \right)\dfrac{\mathbb{E}_{x}\left[ \mathrm{e}^{\Phi_{\lambda }X_{\rho _{\lambda }^{(n-1)}}}\mathbf{1}_{\left\{ \rho _{\lambda}^{(n-1)}<\infty \right\} }\right]}{1-\mathbb{E}\left[ \mathrm{e}^{\Phi_{\lambda }X_{\rho _{\lambda }^{(n-1)}}}\mathbf{1}_{\left\{ \rho _{\lambda}^{(n-1)}<\infty \right\} }\right]} .
\end{equation}
\end{prop}
\begin{remark}
When $n$ tends to infinity, it is possible to approximate the probability of Parisian ruin with fixed delay, that is
\begin{equation}\label{kappar}
\kappa_r = \inf \left\lbrace t > 0 \colon t - g_t > r \right\rbrace ,
\end{equation}
where $g_t = \sup \left\lbrace 0 \leq s \leq t \colon X_s \geq 0 \right\rbrace$, by the the probability of Parisian ruin with Erlang$(n,\lambda=n/r)$ distributed implementation delays (see Bladt et al. \cite{bladtetal2019} and Landriault et al. \cite{landriaultetal2011}).
\end{remark}
\section{Acknowledgements}
The author would like to thank Professor Bin Li for the stimulating discussions on the topic.

\appendix
\section{Fluctuation identities with delays}
In this subsection , we present some of the existing fluctuation identities with delays. When $\e[X_{1}]>0$, we have
\begin{equation}\label{Pruine1}
\p_{x}\left( T_{0}^-<\infty\right)=1-\e_{x}\left[ \me^{-\lambda\Occ_\infty} \right] = 1 - \e[X_{1}] \frac{\Phi_\lambda}{\lambda} Z \left(x,\Phi_\lambda \right).
\end{equation}
We also have the following expression for the the joint Laplace transform of $\left( T_{0}^{-},X_{T_{0}^{-}}\right)$ which has been derived by Albrecher et al. \cite{albrecheretal2016}. 
\begin{lemma}
For $\lambda>0$, $a,b,q,\theta\geq 0$ and $x \leq b$, we have
\begin{equation}
\mathbb{E}_{x}\left[ \mathrm{e}^{-qT_{0}^{-}+\theta X_{T_{0}^{-}}}\ind_{\left\lbrace T_{0}^{-}<\tau^{+}_b \right\rbrace}
\right] =\dfrac{\lambda }{\lambda -\psi _{q}(\theta )}\left( Z_{q}\left(
x,\theta \right) -Z_{q}\left( x,\Phi _{\lambda +q}\right) \dfrac{Z_{q}\left(b,\theta\right) }{Z_{q}\left( b,\Phi _{\lambda +q}\right)}\right).  \label{jointLap1b}
\end{equation}
For $-a\leq x\leq b$,
\begin{equation}\label{jointLap1sda}
\mathbb{E}_{x}\left[ \mathrm{e}^{-q\tau^{+}_b}\ind_{\left\lbrace \tau^{+}_b<T_{0}^{-}\wedge \tau^{-}_{-a} \right\rbrace}\right]=\frac{\mathcal{W}_{x}^{\left(q,\lambda\right) }\left(x+a\right) }{\mathcal{W}_{b}^{\left(q,\lambda\right) }\left(b+a\right) },
\end{equation}
and consequently,
\begin{equation}
\mathbb{E}_{x}\left[ \mathrm{e}^{-q\tau^{+}_b}\ind_{\left\lbrace \tau^{+}_b<T_{0}^{-} \right\rbrace}\right]=\dfrac{Z_{q}\left(x,\Phi _{\lambda +q}\right) }{Z_{q}\left( b,\Phi _{\lambda +q}\right)}.  \label{jointLap1sd}
\end{equation}
\end{lemma}
We also have the following useful identity one can extract from \cite{lkabous2018}.
\begin{lemma}\label{L:L5}
For $a,b,p,q\geq 0$, $\lambda, r,z>0$ and $-a\leq x\leq b$, we have
\begin{multline}
\mathbb{E}_{x}\left[ \mathrm{e}^{-qT_{0}^{-}}W_{p}\left(
X_{T_{0}^{-}}+z\right) \mathbf{1}_{\left\{ T_{0}^{-}<\tau _{b}^{+}\wedge \tau^{-}_{-a} \right\} }%
\right]\\ =\dfrac{\lambda }{p-\left( q+\lambda \right)}\frac{\mathcal{W}_{x}^{\left(q,\lambda\right) }\left(x+a\right) }{\mathcal{W}_{b}^{\left(q,\lambda\right) }\left(b+a\right) }%
\left( \mathcal{W}_{b}^{\left( q,p-q\right) }\left( b+z\right) -\mathcal{W}%
_{b}^{\left( q,\lambda \right) }\left( b+z\right) \right)   \label{LL1} \\
-\dfrac{\lambda }{p-\left( q+\lambda \right) }\left( \mathcal{W}%
_{x}^{\left( q,p-q\right) }\left( x+z\right) -\mathcal{W}_{x}^{\left(
q,\lambda \right) }\left( x+z\right) \right).
\end{multline}
\end{lemma}
\bibliographystyle{alpha}
\bibliography{REFERENCES}

\begin{bibdiv}
\begin{biblist}

\bib{albrecher-ivanovs2017}{article}{
      author={Albrecher, H.},
      author={Ivanovs, J.},
       title={Strikingly simple identities relating exit problems for {L}\'evy
  processes under continuous and {P}oisson observations},
        date={2017},
     journal={Stochastic Process. Appl.},
      volume={127},
      number={2},
       pages={643\ndash 656},
}

\bib{albrecheretal2016}{article}{
      author={Albrecher, H.},
      author={Ivanovs, J.},
      author={Zhou, X.},
       title={Exit identities for {L}\'evy processes observed at {P}oisson
  arrival times},
        date={2016},
        ISSN={1350-7265},
     journal={Bernoulli},
      volume={22},
      number={3},
       pages={1364\ndash 1382},
}

\bib{baurdoux_et_al_2015}{article}{
      author={Baurdoux, E.~J.},
      author={Pardo, J.~C.},
      author={P\'erez, J.~L.},
      author={Renaud, J.-F.},
       title={Gerber-{S}hiu distribution at {P}arisian ruin for {L}\'evy
  insurance risk processes},
        date={2016},
     journal={J. Appl. Probab.},
}

\bib{bertoin1996}{book}{
      author={Bertoin, J.},
       title={L\'evy processes},
   publisher={Cambridge University Press},
        date={1996},
}

\bib{bladtetal2019}{article}{
      author={Bladt, M.},
      author={Nielsen, B.~F.},
      author={Peralta, O.},
       title={Parisian types of ruin probabilities for a class of dependent
  risk-reserve processes},
        date={2019},
     journal={Scandinavian Actuarial Journal},
      volume={2019},
      number={1},
       pages={32\ndash 61},
      eprint={https://doi.org/10.1080/03461238.2018.1483420},
         url={https://doi.org/10.1080/03461238.2018.1483420},
}

\bib{Frostig2019}{article}{
      author={Frostig, E.},
      author={Keren-Pinhasik, A.},
       title={Parisian ruin with erlang delay and a lower bankruptcy barrier},
        date={2019Jan},
        ISSN={1573-7713},
     journal={Methodology and Computing in Applied Probability},
         url={https://doi.org/10.1007/s11009-019-09693-w},
}

\bib{guerin_renaud_2015}{article}{
      author={Gu\'erin, H.},
      author={Renaud, J.-F.},
       title={Joint distribution of a spectrally negative {L}\'evy process and
  its occupation time, with step option pricing in view},
        date={2016},
     journal={Adv. in Appl. Probab.},
}

\bib{guerinrenaud2015}{article}{
      author={Gu\'erin, H.},
      author={Renaud, J.-F.},
       title={On the distribution of cumulative {P}arisian ruin},
        date={2017},
     journal={Insurance Math. Econom.},
      volume={73},
       pages={116\ndash 123},
}

\bib{kuznetsovetal2012}{book}{
      author={Kuznetsov, A.},
      author={Kyprianou, A.~E.},
      author={Rivero, V.},
       title={The theory of scale functions for spectrally negative {L}\'evy
  processes},
   publisher={{L}\'evy Matters - Springer Lecture Notes in Mathematics},
        date={2012},
}

\bib{kyprianou2014}{book}{
      author={Kyprianou, A.~E.},
       title={Fluctuations of {L}\'evy processes with applications -
  {I}ntroductory lectures},
     edition={Second},
      series={Universitext},
   publisher={Springer, Heidelberg},
        date={2014},
}

\bib{lietal2019}{article}{
      author={Landriault, D.},
      author={Li, B.},
      author={Lkabous, M.A.},
       title={On occupation times in the red of {L}\'evy risk models},
        date={submitted},
}

\bib{landriaultetal2011}{article}{
      author={Landriault, D.},
      author={Renaud, J.-F.},
      author={Zhou, X.},
       title={Occupation times of spectrally negative {L}\'evy processes with
  applications},
        date={2011},
     journal={Stochastic Process. Appl.},
      volume={121},
      number={11},
       pages={2629\ndash 2641},
}

\bib{landriaultetal2014}{article}{
      author={Landriault, D.},
      author={Renaud, J.-F.},
      author={Zhou, X.},
       title={An insurance risk model with {P}arisian implementation delays},
        date={2014},
     journal={Methodol. Comput. Appl. Probab.},
      volume={16},
      number={3},
       pages={583\ndash 607},
}

\bib{li-palmowski2017}{article}{
      author={Li, B.},
      author={Palmowski, Z.},
       title={Fluctuations of omega-killed spectrally negative {L}\'evy
  processes},
        date={2017},
        ISSN={0304-4149},
     journal={Stochastic Processes and their Applications},
  url={http://www.sciencedirect.com/science/article/pii/S030441491730279X},
}

\bib{MR3257360}{article}{
      author={Li, Y.},
      author={Zhou, X.},
       title={On pre-exit joint occupation times for spectrally negative
  {L}\'evy processes},
        date={2014},
        ISSN={0167-7152},
     journal={Statist. Probab. Lett.},
      volume={94},
       pages={48\ndash 55},
  url={https://doi-org.proxy.bibliotheques.uqam.ca:2443/10.1016/j.spl.2014.06.023},
}

\bib{lietal2015}{article}{
      author={Li, Y.},
      author={Zhou, X.},
      author={Zhu, N.},
       title={Two-sided discounted potential measures for spectrally negative
  {L}\'evy processes},
        date={2015},
        ISSN={0167-7152},
     journal={Statist. Probab. Lett.},
      volume={100},
       pages={67\ndash 76},
  url={https://doi-org.proxy.bibliotheques.uqam.ca:2443/10.1016/j.spl.2015.01.022},
}

\bib{lkabous2018}{article}{
      author={Lkabous, M.~A.},
       title={A note on {P}arisian ruin under a hybrid observation scheme},
        date={2019},
        ISSN={0167-7152},
     journal={Statistics {\&} Probability Letters},
      volume={145},
       pages={147 \ndash  157},
  url={http://www.sciencedirect.com/science/article/pii/S0167715218303122},
}

\bib{loeffenetal2013}{article}{
      author={Loeffen, R.~L.},
      author={Czarna, I.},
      author={Palmowski, Z.},
       title={Parisian ruin probability for spectrally negative {L}\'evy
  processes},
        date={2013},
     journal={Bernoulli},
      volume={19},
      number={2},
       pages={599\ndash 609},
}

\bib{loeffenetal2017}{article}{
      author={Loeffen, R.~L.},
      author={Palmowski, Z.},
      author={Surya, B.~A.},
       title={Discounted penalty function at {P}arisian ruin for {L}\'evy
  insurance risk process},
        date={2017},
     journal={Insurance Math. Econom.},
}

\bib{loeffenetal2014}{article}{
      author={Loeffen, R.~L.},
      author={Renaud, J.-F.},
      author={Zhou, X.},
       title={Occupation times of intervals until first passage times for
  spectrally negative {L}\'evy processes},
        date={2014},
     journal={Stochastic Process. Appl.},
      volume={124},
      number={3},
       pages={1408\ndash 1435},
}

\bib{surya2008}{article}{
      author={Surya, B.~A.},
       title={Evaluating scale functions of spectrally negative lévy
  processes},
        date={2008},
        ISSN={00219002},
     journal={Journal of Applied Probability},
      volume={45},
      number={1},
       pages={135\ndash 149},
}

\end{biblist}
\end{bibdiv}

\end{document}